\documentclass[10pt]{article}
\usepackage{amssymb,enumerate}
\usepackage{amsmath,amsfonts}
\usepackage{amsthm}
\usepackage{latexsym}
\usepackage{mathrsfs}
\usepackage{color}
\usepackage{microtype}
\usepackage{geometry}
\usepackage{pifont}
\usepackage{algpseudocode}
\usepackage{algorithm}
\usepackage{algorithmicx}
\usepackage{algpseudocode}

\usepackage{multirow}

\allowdisplaybreaks[3]

\DeclareMathOperator{\cC}{\ensuremath{\mathcal{C}}}
\DeclareMathOperator{\cO}{\ensuremath{\mathcal{O}}}
\DeclareMathOperator{\cS}{\ensuremath{\mathcal{S}}}
\DeclareMathOperator{\cL}{\ensuremath{\mathcal{L}}}

\DeclareMathOperator{\cA}{\ensuremath{\mathcal{A}}}

\DeclareMathOperator{\st}{\ensuremath{\mathrm{s.t.}}}

\DeclareMathOperator{\bR}{\ensuremath{\mathbb{R}}}
\DeclareMathOperator{\bK}{\ensuremath{\mathbb{K}}}

\DeclareMathOperator{\low}{\ensuremath{\mathrm{low}}}

\DeclareMathOperator{\sgn}{\ensuremath{\mathrm{sgn}}}

\DeclareMathOperator{\init}{\ensuremath{\mathrm{init}}}

\DeclareMathOperator{\hi}{\ensuremath{\mathrm{hi}}}

\newtheorem{lemma}{Lemma}[section]

\newtheorem{assumption}{Assumption}[section]
\newtheorem{definition}{Definition}[section]
\newtheorem{theorem}{Theorem}[section]
\newtheorem{remark}{Remark}[section]

\newcommand{\beq}{\begin{equation}}
	\newcommand{\eeq}{\end{equation}}
\newcommand{\beqa}{\begin{eqnarray}}
	\newcommand{\eeqa}{\end{eqnarray}}
\newcommand{\beqas}{\begin{eqnarray*}}
	\newcommand{\eeqas}{\end{eqnarray*}}
\newcommand{\ba}{\begin{array}}
	\newcommand{\ea}{\end{array}}
\newcommand{\bi}{\begin{statementize}}
	\newcommand{\ei}{\end{statementize}}

\def\bdelta{{\delta}}
\def\bT{{\mathbb{T}}}

\def\scL{{\mathscr L}}

\def\fh{{f_{\rm hi}}}
\def\fl{{f_{\rm low}}}
\def\Fh{{F_{\hi}}}
\def\Fl{{F_{\rm low}}}

\def\nn{{\nonumber}}
\def\tdc{{\tilde \delta_c}}
\def\tdf{{\tilde \delta_f}}

\def\hx{{\hat x}}
\def\hlambda{{\hat \lambda}}

\title{A Newton-CG based augmented Lagrangian method for finding a second-order stationary point of nonconvex equality constrained optimization with complexity guarantees}

\author{
	Chuan He\thanks{Department of Industrial and Systems Engineering, University of Minnesota, USA (email: {\tt he000233@umn.edu}, {\tt zhaosong@umn.edu}). The work of the second author was partially supported by NSF Award IIS-2211491.}
	\and
	Zhaosong Lu\footnotemark[1]\\
	\and
	Ting Kei Pong\thanks{Department of Applied Mathematics, the Hong Kong Polytechnic University, Hong Kong, People's Republic of China (email: \texttt{tk.pong@polyu.edu.hk}). The work of this author was partially supported by a Research Scheme of the Research Grants Council of Hong Kong SAR, China (Project No. T22-504/21R).}
}
\date{April 10, 2022 (Revised: September 22, 2022; 
December 31, 2022)}

\begin{document}
\maketitle
\begin{abstract}
In this paper we consider finding a second-order stationary point (SOSP) of nonconvex equality constrained optimization when a nearly feasible point is known. In particular, we first propose a new Newton-CG method for finding an approximate SOSP of unconstrained optimization and show that it enjoys a substantially better complexity than the Newton-CG method \cite{RNW18}. We then propose a Newton-CG based augmented Lagrangian (AL) method for finding an approximate SOSP of nonconvex equality constrained optimization, in which the proposed Newton-CG method is used as a subproblem solver. We show that under a generalized linear independence constraint qualification (GLICQ), our AL method enjoys a total inner iteration complexity of $\widetilde{\cO}(\epsilon^{-7/2})$ and an operation complexity of $\widetilde{\cO}(\epsilon^{-7/2}\min\{n,\epsilon^{-3/4}\})$ for finding an $(\epsilon,\sqrt{\epsilon})$-SOSP of nonconvex equality constrained optimization with high probability, which are significantly better than the ones achieved by the proximal AL method \cite{XW19}. Besides, we show that it has a total inner iteration complexity of $\widetilde{\cO}(\epsilon^{-11/2})$ and an operation complexity of $\widetilde{\cO}(\epsilon^{-11/2}\min\{n,\epsilon^{-5/4}\})$ when the GLICQ does not hold. To the best of our knowledge, all the complexity results obtained in this paper are new for finding an approximate SOSP of nonconvex equality constrained optimization with high probability. Preliminary numerical results also demonstrate the superiority of our proposed methods 
over the ones in \cite{RNW18,XW19}.

\end{abstract}


\medskip

\section{Introduction}
In this paper we consider nonconvex equality constrained optimization problem
\beq\label{model:equa-cnstr}
\min_{x\in \bR^n}\ f(x)\quad \st\ c(x)=0,
\eeq
where $f:\bR^n\to\bR$ and $c:\bR^n\to\bR^m$ are twice continuously differentiable, and we assume that problem~\eqref{model:equa-cnstr} has at least one optimal solution. Since \eqref{model:equa-cnstr} is a nonconvex optimization problem, it may have many local but non-global minimizers and finding its global minimizer is generally NP-hard. A first-order stationary point (FOSP) of it is usually found in practice instead. Nevertheless,  a mere FOSP may sometimes not suit our needs and a  \emph{second-order stationary point} (SOSP) needs to be sought. For example, in the context of linear semidefinite programming (SDP), a powerful approach to solving it is by solving an equivalent nonconvex equality constrained optimization problem \cite{BuMo03-1,BuMo05-1}. It was shown in \cite{BuMo05-1,BVB16} that under some mild conditions an SOSP of the latter problem can yield an optimal solution of  the linear SDP, while a mere FOSP generally cannot.  It is therefore important to find an SOSP of problem~\eqref{model:equa-cnstr}.

In recent years, numerous methods with complexity guarantees have been developed for finding an approximate SOSP of several types of nonconvex optimization. For example,  cubic regularized Newton methods \cite{NP06,CGT11b,AABHM17,CD17}, accelerated gradient methods \cite{CDHS17,CDHS18}, trust-region methods \cite{CDCS21trust,CRS17,MR17}, quadratic regularization method \cite{BM17}, second-order line-search method \cite{RW18}, and Newton-conjugate gradient (Newton-CG) method \cite{RNW18} were developed for nonconvex unconstrained optimization. In addition, interior-point method \cite{BCY15} and log-barrier method \cite{NW19} were proposed for nonconvex optimization with sign constraints. The interior-point method \cite{BCY15} was also generalized in \cite{HLY18} to solve nonconvex optimization with sign constraints and additional linear equality constraints. Furthermore, a projected gradient descent method with random perturbations was proposed in \cite{LRYHH20} for nonconvex optimization with linear inequality constraints. Iteration complexity was established for these methods for finding an approximate SOSP. Besides, operation complexity measured by the amount of fundamental operations such as gradient evaluations and matrix-vector products was also studied in \cite{AABHM17,CDHS17,CDCS21trust,JGNKJ17,CDHS18,RW18,CD17,RNW18}.

Several methods including trust-region methods \cite{BSS87,CLY02}, sequential quadratic programming method \cite{BL95}, two-phase method \cite{BGMST16,CGT19eq,CM19poly} and augmented Lagrangian (AL) type methods  \cite{AHRS17second,BHR18AL,S19iAL,XW19} were proposed for finding an SOSP of problem \eqref{model:equa-cnstr}.
However, only a few of them have \emph{complexity guarantees} for finding an approximate SOSP of \eqref{model:equa-cnstr}. In particular, the inexact AL method \cite{S19iAL} has a worst-case complexity in terms of the number of calls to a second-order oracle. Yet its  operation complexity, measured by the amount of fundamental operations such as gradient evaluations and Hessian-vector products, is unknown.
To the best of our knowledge, the proximal AL method in \cite{XW19} appears to be the only existing method that enjoys a worst-case complexity for finding an approximate SOSP of \eqref{model:equa-cnstr} in terms of fundamental operations.  In this method,  given an iterate $x^k$ and a multiplier estimate $\lambda^k$ at the $k$th iteration, the next iterate $x^{k+1}$ is obtained by finding an approximate stochastic SOSP of the proximal AL subproblem:
\begin{equation*}
\min_{x\in \bR^n}\ \cL(x,\lambda^k;\rho)+\beta\|x-x^k\|^2/2
\end{equation*}
for some suitable positive $\rho$ and $\beta$ using a Newton-CG method proposed in \cite{RNW18}, where $\cL$ is the AL function of \eqref{model:equa-cnstr} defined as
\begin{equation*}\label{AL-func}
\cL(x,\lambda;\rho):=f(x)+\lambda^Tc(x)+\rho\|c(x)\|^2/2.
\end{equation*}
Then the multiplier estimate is updated using the classical scheme, i.e., $\lambda^{k+1}=\lambda^k+\rho c(x^{k+1})$ (e.g., see \cite{H69,R93}). The authors of \cite{XW19} studied the worst-case complexity of their proximal AL method  including: (i) \textit{total inner iteration complexity}, which measures the total number of iterations of the Newton-CG method  \cite{RNW18} performed in their method; (ii) \textit{operation complexity}, which measures the total number of gradient evaluations and matrix-vector products involving the Hessian of the AL function that are evaluated in their method. Under some suitable assumptions, including that a generalized linear independence constraint qualification (GLICQ) holds at all iterates, it was established in \cite{XW19} that their proximal AL method enjoys a total inner iteration complexity of $\widetilde{\cO}(\epsilon^{-11/2})$  and an operation complexity of $\widetilde{\cO}(\epsilon^{-11/2}\min\{n,\epsilon^{-3/4}\})$ for finding an $(\epsilon,\sqrt{\epsilon})$-SOSP of problem \eqref{model:equa-cnstr} with high probability.\footnote{In fact, a total inner iteration complexity of $\widetilde{\cO}(\epsilon^{-7})$  and an operation complexity of $\widetilde{\cO}(\epsilon^{-7}\min\{n,\epsilon^{-1}\})$ were established in \cite{XW19} for finding an $(\epsilon,\epsilon)$-SOSP of problem \eqref{model:equa-cnstr} with high probability; see  \cite[Theorem~4(ii), Corollary~3(ii), Theorem~5]{XW19}.  Nonetheless, they can be modified to obtain the aforementioned complexity for finding an $(\epsilon,\sqrt{\epsilon})$-SOSP of \eqref{model:equa-cnstr} with high probability.}
Yet, there is a big gap between these complexities and
the iteration complexity of $\widetilde{\cO}(\epsilon^{-3/2})$ and the operation complexity of $\widetilde{\cO}(\epsilon^{-3/2}\min\{n,\epsilon^{-1/4}\})$ that are achieved by the methods in \cite{AABHM17,CDHS18,RW18,RNW18} for finding an $(\epsilon,\sqrt{\epsilon})$-SOSP of  nonconvex unconstrained optimization with high probability, which is a special case of \eqref{model:equa-cnstr} with $c\equiv 0$. Also, there is a lack of complexity guarantees for this proximal AL method when the GLICQ does not hold.
It shall be mentioned that Newton-CG based AL methods were also developed for efficiently solving various convex optimization problems (e.g., see \cite{YST15sdpnal,ZST10NCGAL}), though their complexities remain unknown.


In this paper we propose a Newton-CG based AL method for finding an approximate SOSP of problem~\eqref{model:equa-cnstr} with high probability, and study its worst-case complexity with and without the assumption of a GLICQ. In particular, we show that this method enjoys a total inner iteration complexity of $\widetilde{\cO}(\epsilon^{-7/2})$ and an operation complexity of $\widetilde{\cO}(\epsilon^{-7/2}\min\{n,\epsilon^{-3/4}\})$ for finding a stochastic $(\epsilon,\sqrt{\epsilon})$-SOSP of \eqref{model:equa-cnstr} under the GLICQ, which are significantly better than the aforementioned ones achieved by the proximal AL method in \cite{XW19}.  Besides, when the GLICQ does not hold, we show that it has a total inner iteration complexity of $\widetilde{\cO}(\epsilon^{-11/2})$ and an operation complexity of $\widetilde{\cO}(\epsilon^{-11/2}\min\{n,\epsilon^{-5/4}\})$ for finding a stochastic $(\epsilon,\sqrt{\epsilon})$-SOSP of \eqref{model:equa-cnstr},
which fills the research gap in this topic.  Specifically, our AL method (Algorithm~\ref{alg:2nd-order-AL-nonconvex}) proceeds in the following manner. Instead of directly solving problem~\eqref{model:equa-cnstr}, it solves a perturbed problem of \eqref{model:equa-cnstr} with $c$ replaced by its perturbed counterpart $\tilde{c}$ constructed by using a nearly feasible point of \eqref{model:equa-cnstr}
(see \eqref{model:equa-cnstr-pert} for details). At the $k$th iteration, an approximate stochastic SOSP $x^{k+1}$ of the AL subproblem of this perturbed problem is found by our newly proposed Newton-CG method (Algorithm \ref{alg:NCG}) for a penalty parameter $\rho_k$ and a truncated Lagrangian multiplier
$\lambda^k$, which results from projecting onto a Euclidean ball the standard multiplier estimate $\tilde{\lambda}^k$ obtained by the classical scheme $\tilde{\lambda}^k=\lambda^{k-1}+\rho_k \tilde{c}(x^{k})$.\footnote{The $\lambda^k$ obtained by projecting $\tilde{\lambda}^k$ onto a compact set is also called a safeguarded Lagrangian multiplier in the relevant literature \cite{BM14,KS17example,BM20}, which has been shown to enjoy many practical and theoretical advantages (see \cite{BM14} for discussions).} The penalty parameter $\rho_{k+1}$ is then updated by the following practical scheme (e.g., see \cite[Section~4.2]{B97}):
\begin{equation*}\label{rhok-update-rule}
\rho_{k+1}=\left\{\begin{array}{cl}
r\rho_k& \text{if } \|\tilde{c}(x^{k+1})\|>\alpha\|\tilde{c}(x^k)\|,\\
\rho_k&\text{otherwise}
\end{array}\right.
\end{equation*}
for some $r>1$ and $\alpha\in(0,1)$. It shall be mentioned that in contrast with the classical AL method, our method has two distinct features:
(i) the values of the AL function along the iterates are bounded from above;
(ii) the multiplier estimates associated with the AL subproblems are bounded. In addition, to solve the AL subproblems with better complexity guarantees, we propose a variant of the Newton-CG method in \cite{RNW18} for finding an approximate stochastic SOSP of unconstrained optimization, whose complexity has significantly less dependence on the Lipschitz constant of the Hessian of the objective than that of the Newton-CG method in \cite{RNW18},  while improving or retaining the same order of dependence on tolerance parameter. Given that such a Lipschitz constant is typically large for the AL subproblems, our Newton-CG method (Algorithm \ref{alg:NCG}) is a much more favorable subproblem solver
than the Newton-CG method in \cite{RNW18} that is used in the proximal AL method in \cite{XW19}
from theoretical complexity perspective.


The main contributions of this paper are summarized below.

\begin{itemize}
\item We propose a new Newton-CG method for finding an approximate SOSP of  unconstrained optimization and show that it enjoys an iteration and operation complexity with a {\it quadratic} dependence on the Lipschitz constant of the Hessian of the objective that improves the {\it cubic} dependence achieved by the Newton-CG method in \cite{RNW18}, while improving or retaining the same order of dependence on tolerance parameter.  In addition, our complexity results are established under the assumption that the Hessian of the objective is Lipschitz continuous in a convex neighborhood of a level set of the objective. This assumption is weaker than the one commonly imposed for the Newton-CG method in \cite{RNW18} and some other methods (e.g., \cite{BM17,CRS17}) that the Hessian of the objective is Lipschitz continuous  in a convex set containing this neighborhood and also  {\it all the trial points} arising in the line search or trust region steps of the methods (see Section~\ref{sec:sbpb-solver} for more detailed discussion). 

\item
We propose a Newton-CG based AL method for finding an approximate SOSP of nonconvex equality constrained optimization~\eqref{model:equa-cnstr} with high probability, and study its worst-case complexity with and without the assumption of a GLICQ. Prior to our work, there was no complexity study on finding an approximate SOSP of problem \eqref{model:equa-cnstr} without imposing a GLICQ. Besides, under the GLICQ and some other suitable assumptions, we show that our method enjoys a total inner iteration complexity of $\widetilde{\cO}(\epsilon^{-7/2})$ and an operation complexity of $\widetilde{\cO}(\epsilon^{-7/2}\min\{n,\epsilon^{-3/4}\})$ for finding an $(\epsilon,\sqrt{\epsilon})$-SOSP of \eqref{model:equa-cnstr} with high probability, which are significantly better than the respective complexity of $\widetilde{\cO}(\epsilon^{-11/2})$ and $\widetilde{\cO}(\epsilon^{-11/2}\min\{n,\epsilon^{-3/4}\})$ achieved by the proximal AL method in \cite{XW19}.
To the best of our knowledge, all the complexity results obtained in this paper are new for finding an approximate SOSP of nonconvex equality constrained optimization with high probability.
\end{itemize}

For ease of comparison, we summarize in Table~\ref{table:cmplx} the total inner iteration and operation complexity of our AL method and the proximal AL method in \cite{XW19} for finding a stochastic $(\epsilon,\sqrt{\epsilon})$-SOSP of problem \eqref{model:equa-cnstr} with or without assuming GLICQ.

\begin{table}[ht]
\caption{Total inner iteration and operation complexity of finding a stochastic $(\epsilon,\sqrt{\epsilon})$-SOSP of \eqref{model:equa-cnstr}.} 
\centering 
\vspace{2mm}
\begin{tabular}{c | c | c | c} 
\hline
 Method &GLICQ &Total inner iteration complexity & Operation complexity\\ [0.5ex]
\hline
Proximal AL method \cite{XW19}&\ding{51} & $\widetilde{\cO}(\epsilon^{-11/2})$ & $\widetilde{\cO}(\epsilon^{-11/2}\min\{n,\epsilon^{-3/4}\})$\\
Proximal AL method \cite{XW19} &\ding{55}&unknown & unknown \\
Our AL method &\ding{51} & $\widetilde{\cO}(\epsilon^{-7/2})$ & $\widetilde{\cO}(\epsilon^{-7/2}\min\{n,\epsilon^{-3/4}\})$\\
Our AL method &\ding{55}&$\widetilde{\cO}(\epsilon^{-11/2})$ & $\widetilde{\cO}(\epsilon^{-11/2}\min\{n,\epsilon^{-5/4}\})$\\
\hline
\end{tabular}
\label{table:cmplx}
\end{table}


It shall be mentioned that there are many works other than \cite{XW19} studying complexity of AL methods for nonconvex constrained optimization. However, they aim to find an approximate FOSP rather than SOSP of the problem (e.g., see \cite{HHZ17prox,GY19,BM20,MMK20ipAL,LPLLX21rAL}). Since our main focus is on the complexity of finding an approximate SOSP by AL methods, we do not include them in the above table for comparison.

The rest of this paper is organized as follows. In Section \ref{sec:not-pre}, we introduce some notation and optimality conditions. In Section~\ref{sec:sbpb-solver}, we propose a Newton-CG method  for unconstrained optimization and study its worst-case complexity. In Section \ref{sec:AL-method}, we propose a Newton-CG based AL method for \eqref{model:equa-cnstr} and study its worst-case complexity. 
We present numerical results and the proof of the main results in Sections~\ref{sec:nr} and \ref{sec:proof},  respectively. In Section~\ref{sec:cr}, we discuss some future research directions.


\section{Notation and preliminaries}\label{sec:not-pre}
Throughout this paper, we let $\bR^n$ denote the $n$-dimensional Euclidean space. We use $\|\cdot\|$ to denote the Euclidean norm of a vector or the spectral norm of a matrix. For a real symmetric matrix $H$, we use $\lambda_{\min}(H)$ to denote its minimum eigenvalue.
The Euclidean ball centered at the origin with radius $R\ge0$ is denoted by $\mathcal{B}_R:=\{x:\|x\|\le R\}$, and we use $\Pi_{\mathcal{B}_R}(v)$ to denote the Euclidean projection of a vector $v$ onto $\mathcal{B}_R$. For a given finite set $\cA$, we let $|\cA|$ denote its cardinality. For any $s\in \mathbb{R}$, we let ${\rm sgn}(s)$ be $1$ if $s \ge 0$ and let it be $-1$ otherwise. In addition, $\widetilde{\cO}(\cdot)$ represents $\cO(\cdot)$ with logarithmic terms omitted. 
	
Suppose that $x^*$ is a local minimizer of problem \eqref{model:equa-cnstr} and the linear independence constraint qualification holds at $x^*$, i.e.,
$\nabla c(x^*):=[\nabla c_1(x^*) \ \nabla c_2(x^*) \ \cdots \ \nabla c_m(x^*)]$
has full column rank. Then there exists a Lagrangian multiplier $\lambda^*\in\bR^m$ such that 
\begin{eqnarray} 
&&\nabla f(x^*)+\nabla c(x^*)\lambda^*=0,\label{exact-1st-opt}\\
&& d^T\left(\nabla^2 f(x^*)+\sum_{i=1}^m\lambda_i^*\nabla^2 c_i(x^*)\right)d\ge0,\quad \forall d\in\cC(x^*),\label{exact-2nd-opt}
\end{eqnarray}
where $\cC(\cdot)$ is defined as 
\begin{equation}\label{def:critical-cone} 
\cC(x):=\{d\in\bR^n:\nabla c(x)^T d=0\}.
\end{equation}
The relations \eqref{exact-1st-opt} and \eqref{exact-2nd-opt} are respectively known as the first- and second-order optimality conditions for \eqref{model:equa-cnstr} in the literature (e.g., see \cite{NW06}). Note that it is in general impossible to find a point that exactly satisfies \eqref{exact-1st-opt} and \eqref{exact-2nd-opt}. Thus, we are instead interested in finding a point that satisfies their approximate counterparts. In particular, we introduce the following definitions of an approximate first-order stationary point (FOSP) and second-order stationary point (SOSP), which are similar to those considered in \cite{AHRS17second,BHR18AL,XW19}.
 The rationality of them can be justified by the study of the sequential optimality conditions for constrained optimization \cite{AHM11soc,AHRS17second}.

\begin{definition}[{{\bf $\epsilon_1$-first-order stationary point}}]\label{def:FOSP}
Let $\epsilon_1>0$. We say that $x\in\bR^n$ is an $\epsilon_1$-first-order stationary point {\rm ($\epsilon_1$-FOSP)} of problem \eqref{model:equa-cnstr} if it, together with some $\lambda\in\bR^m$, satisfies 
\begin{equation}\label{optcond:1st-equa-cnstr} 
\|\nabla f(x)+\nabla c(x)\lambda\|\le \epsilon_1,\quad \|c(x)\|\le \epsilon_1.
\end{equation}
\end{definition}
	
\begin{definition}[{{\bf $(\epsilon_1,\epsilon_2)$-second-order stationary point}}]\label{def:SOSP}
Let $\epsilon_1,\epsilon_2>0$. We say that $x\in\bR^n$ is an $(\epsilon_1,
\epsilon_2)$-second-order stationary point {\rm (($\epsilon_1,
\epsilon_2$)-SOSP)} of problem \eqref{model:equa-cnstr} if it, together with some $\lambda\in\bR^m$, satisfies \eqref{optcond:1st-equa-cnstr} and additionally 
\begin{equation}\label{optcond:2nd-equa-cnstr} 
d^T\left(\nabla^2 f(x)+\sum_{i=1}^m\lambda_i\nabla^2 c_i(x)\right)d\ge-\epsilon_2\|d\|^2,\quad \forall d\in\cC(x),
\end{equation}
where $\cC(\cdot)$ is defined as in \eqref{def:critical-cone}.
\end{definition}

\section{A Newton-CG method for unconstrained optimization}\label{sec:sbpb-solver}

In this section we propose a variant of Newton-CG method \cite[Algorithm~3]{RNW18}  for finding an approximate SOSP of a class of unconstrained optimization problems, which will be used as a subproblem solver for the AL method proposed in the next section. In particular, we consider an unconstrained optimization problem 
\begin{equation}\label{unconstrained-prob} 
\min_{x\in\bR^n}\ F(x),
\end{equation}
where the function $F$ satisfies the following assumptions.

\begin{assumption}\label{asp:NCG-cmplxity}
\begin{enumerate}[{\rm (a)}]
\item The level set $\scL_F(u^0):=\{x: F(x)\le F(u^0)\}$ is compact for some $u^0\in \bR^n$.
\item The function $F$ is twice Lipschitz continuously differentiable in
a convex open neighborhood, denoted by $\Omega$, of $\scL_F(u^0)$,
that is, there exists $L_{H}^F>0$ such that
\begin{equation}\label{F-Hess-Lip}
\|\nabla^2 F(x)-\nabla^2 F(y)\|\le L_{H}^F\|x-y\|,\quad \forall x,y\in \Omega.
\end{equation}
\end{enumerate}
\end{assumption}

By Assumption \ref{asp:NCG-cmplxity}, there exist $\Fl\in \bR$, $U_g^F>0$ and $U_H^F>0$ such that
\begin{equation}\label{lwbd-Hgupbd}
F(x)\ge \Fl,\quad \|\nabla F(x)\|\le U_g^F,\quad \|\nabla^2 F(x)\|\le U_H^F, \quad \forall x\in\scL_F(u^0).
\end{equation}

Recently, a Newton-CG method \cite[Algorithm~3]{RNW18}  was developed to find an approximate stochastic SOSP of problem \eqref{unconstrained-prob}, which is not only easy to implement but also enjoys a nice feature that the main computation consists only of gradient evaluations and Hessian-vector products associated with the function $F$. Under the assumption that $\nabla^2 F$ is Lipschitz continuous in a convex open set containing $\scL_F(u^0)$ and also {\it all the trial points} arising in the line search steps of this method (see \cite[Assumption~2]{RNW18}), it was established in \cite[Theorem~4, Corollary~2]{RNW18} that the iteration and operation complexity of this method  for finding a stochastic $(\epsilon_g,\epsilon_H)$-SOSP of \eqref{unconstrained-prob} (namely, a point $x$ satisfying $\|\nabla F(x)\|\le\epsilon_g$ deterministically and $\lambda_{\min}(\nabla^2 F(x))\ge-\epsilon_H$ with high probability) are 
\begin{equation}\label{NCG-cmplx} 
\cO((L_H^F)^3\max\{\epsilon_g^{-3}\epsilon_H^3,\epsilon_H^{-3}\})\quad \text{and}\quad \widetilde{\cO}((L_H^F)^3\max\{\epsilon_g^{-3}\epsilon_H^3,\epsilon_H^{-3}\}\min\{n,(U_{H}^F/\epsilon_H)^{1/2}\}),
\end{equation}
respectively, where $\epsilon_g, \epsilon_H \in (0,1)$ are prescribed tolerances.
\emph{Yet, this assumption can be hard to check} because these trial points are \emph{unknown} before the method terminates and moreover the distance between the origin and them depends on the tolerance $\epsilon_H$ in $\cO(\epsilon_H^{-1})$ (see \cite[Lemma~3]{RNW18}).
In addition, as seen from \eqref{NCG-cmplx}, iteration and operation complexity of the Newton-CG method in \cite{RNW18} depend {\it cubically} on $L^F_H$.  Notice that $L^F_H$ can sometimes be very large. For example,  the AL subproblems arising in Algorithm \ref{alg:2nd-order-AL-nonconvex} have $L^F_H=\cO(\epsilon_1^{-2})$ or $\cO(\epsilon_1^{-1})$, where $\epsilon_1 \in (0,1)$ is a prescribed tolerance for problem  \eqref{model:equa-cnstr} (see Section \ref{sec:AL-method}). The cubic dependence on $L^F_H$ makes such a Newton-CG method not appealing as an AL subproblem solver from theoretical complexity perspective.

In the rest of this section, we propose a variant of the Newton-CG method \cite[Algorithm~3]{RNW18} and show that under Assumption \ref{asp:NCG-cmplxity}, it enjoys an iteration and operation complexity of 
\begin{equation}\label{NCG-cmplx-1} 
\cO((L_H^F)^2\max\{\epsilon_g^{-2}\epsilon_H,\epsilon_H^{-3}\})\quad \text{and}\quad \widetilde{\cO}((L_H^F)^2\max\{\epsilon_g^{-2}\epsilon_H,\epsilon_H^{-3}\}\min\{n,(U_{H}^F/\epsilon_H)^{1/2}\}),
\end{equation}
for finding a stochastic $(\epsilon_g,\epsilon_H)$-SOSP of problem \eqref{unconstrained-prob},
respectively. These complexities are substantially superior to those in \eqref{NCG-cmplx} achieved by the Newton-CG method in \cite{RNW18}.
Indeed, the complexities in \eqref{NCG-cmplx-1} depend quadratically on $L^F_H$,  while those in \eqref{NCG-cmplx} depend cubically on $L^F_H$. In addition, it can be verified that they improve or retain the order of dependence on $\epsilon_g$ and $\epsilon_H$ given in \eqref{NCG-cmplx}. 

\subsection{Main components of a Newton-CG method}\label{sbsc:main-cmpnts}

In this subsection we briefly discuss two main components of the Newton-CG method in \cite{RNW18}, which will be used to propose a variant of this method for finding an approximate stochastic SOSP of problem \eqref{unconstrained-prob} in the next subsection.

The first main component of the Newton-CG method in \cite{RNW18} is a {\it capped CG method} \cite[Algorithm~1]{RNW18}, which is a modified CG method, for solving a possibly indefinite linear system 
\begin{equation}\label{indef-sys} 
(H+2\varepsilon I)d=-g,
\end{equation}
where $0\neq g\in\bR^n$, $\varepsilon>0$, and $H\in\bR^{n\times n}$ is a symmetric matrix. This capped CG method terminates within a finite number of iterations. It outputs either an approximate solution $d$ to \eqref{indef-sys} such that $\|(H+2\varepsilon I)d+g\|\le\widehat{\zeta}\|g\|$ and $d^THd\ge-\varepsilon\|d\|^2$ for some $\widehat{\zeta}\in(0,1)$ or a sufficiently negative curvature direction $d$ of $H$ with $d^THd<-\varepsilon\|d\|^2$. The second main component of the Newton-CG method in \cite{RNW18} is a minimum eigenvalue oracle that either produces a sufficiently negative curvature direction $v$ of $H$ with $\|v\|=1$ and  $v^THv\le-\varepsilon/2$ or certifies that $\lambda_{\min}(H)\ge-\varepsilon$ holds with high probability. For ease of reference, we present these two components in Algorithms~\ref{alg:capped-CG}  and \ref{pro:meo} in Appendices~\ref{appendix:capped-CG} and \ref{appendix:meo}, respectively.

\begin{algorithm}[h]
\caption{A Newton-CG method for problem \eqref{unconstrained-prob}}
\label{alg:NCG}
{\footnotesize
\begin{algorithmic}
\State \noindent\textit{Input}: Tolerances $\epsilon_g,\epsilon_H\in(0,1)$, backtracking ratio $\theta\in(0,1)$, starting point $u^0$, CG-accuracy parameter $\zeta\in(0,1)$, line-search parameter $\eta\in(0,1)$, probability parameter $\delta\in(0,1)$.\\
Set $x^0=u^0$;
\For{$t=0,1,2,\ldots$}
\If{$\|\nabla F(x^t)\|>\epsilon_g$}
\State Call Algorithm~\ref{alg:capped-CG} with $H=\nabla^2 F(x^t)$, $\varepsilon=\epsilon_H$, $g=\nabla F(x^t)$, accuracy parameter $\zeta$, and $U=0$ to obtain outputs $d$,
\State d$\_$type;
\If{d$\_$type=NC}
\begin{equation}\label{dk-nc}
d^t\leftarrow -\sgn(d^T\nabla F(x^t))\frac{|d^T\nabla^2 F(x^t) d|}{\|d\|^3}d;
\end{equation}
\Else\ \{d$\_$type=SOL\}
\begin{equation}\label{dk-sol}
d^t\leftarrow d;
\end{equation}
\EndIf
\State Go to {\bf Line Search};
\Else
\State Call Algorithm~\ref{pro:meo} with $H=\nabla^2 F(x^t)$, $\varepsilon=\epsilon_H$, and probability parameter $\delta$;
\If{Algorithm~\ref{pro:meo} certifies that $\lambda_{\min}(\nabla^2 F(x^t))\ge-\epsilon_H$}
\State Output $x^t$ and terminate;
\Else\ \{Sufficiently negative curvature direction $v$ returned by Algorithm~\ref{pro:meo}\}
\State Set d$\_$type=NC and
\begin{equation}\label{dk-2nd-nc}
d^t\leftarrow -\sgn(v^T\nabla F(x^t))|v^T\nabla^2F(x^t)v|v;
\end{equation}
\State Go to {\bf Line Search};
\EndIf
\EndIf
\State{\bf Line Search:}
\If{d$\_$type=SOL}
\State Find $\alpha_t=\theta^{j_t}$, where $j_t$ is the smallest nonnegative integer $j$ such that
\begin{equation}\label{ls-sol}
F(x^t+\theta^jd^t)<F(x^t)-\eta\epsilon_H\theta^{2j}\|d^t\|^2;
\end{equation}
\Else\ \{d$\_$type=NC\}
\State Find $\alpha_t=\theta^{j_t}$, where $j_t$ is the smallest nonnegative integer $j$ such that
\begin{equation}\label{ls-nc}
F(x^t+\theta^jd^t)<F(x^t)-\eta\theta^{2j}\|d^t\|^3/2;
\end{equation}
\EndIf
\State $x^{t+1}=x^t+\alpha_td^t$;
\EndFor
\end{algorithmic}
}
\end{algorithm}

\subsection{A Newton-CG method for problem \eqref{unconstrained-prob}}

In this subsection we propose a Newton-CG method in Algorithm~\ref{alg:NCG}, which is a variant of the Newton-CG method \cite[Algorithm~3]{RNW18},  for finding an approximate stochastic SOSP of problem \eqref{unconstrained-prob}.

Our Newton-CG method (Algorithm~\ref{alg:NCG}) follows the same framework as \cite[Algorithm~3]{RNW18}. In particular, at each iteration, if the gradient of $F$ at the current iterate is not desirably small, then the capped CG method (Algorithm~\ref{alg:capped-CG}) is called to solve a damped Newton system for obtaining a descent direction and a subsequent line search along this direction results in a sufficient reduction on $F$.  Otherwise, the current iterate is already an approximate first-order stationary point of \eqref{unconstrained-prob}, and the minimum eigenvalue oracle (Algorithm~\ref{pro:meo}) is then called, which either produces a sufficiently negative curvature direction for $F$ and a subsequent line search along this direction results in a sufficient reduction on $F$, or certifies that the current iterate is an approximate SOSP of \eqref{unconstrained-prob} with high probability and terminates the algorithm. More details about this framework can be found in \cite{RNW18}.

Despite sharing the same framework, our Newton-CG method and \cite[Algorithm~3]{RNW18} use different line search criteria. Indeed, our Newton-CG method uses a hybrid line search criterion adopted from \cite{XW21}, which is a combination of the quadratic descent criterion \eqref{ls-sol} and the cubic descent criterion \eqref{ls-nc}. Specifically,  it uses the quadratic descent criterion \eqref{ls-sol} when the search direction is of type `SOL'. On the other hand, it uses the cubic descent criterion \eqref{ls-nc} when the search direction is of type `NC'.\footnote{SOL and NC stand for ``approximate solution'' and ``negative curvature'', respectively.}
In contrast, the Newton-CG method in \cite{RNW18} always uses a cubic descent criterion regardless of the type of search directions. As observed from Theorem \ref{thm:NCG-iter-oper-cmplxity} below, our Newton-CG method achieves an iteration and operation complexity given in \eqref{NCG-cmplx-1}, which are
 superior to those in \eqref{NCG-cmplx} achieved by \cite[Algorithm~3]{RNW18} in terms of the order dependence on $L^F_H$, while improving or retaining the order of dependence on $\epsilon_g$ and $\epsilon_H$ as given in \eqref{NCG-cmplx}. Consequently, our Newton-CG method is more appealing than \cite[Algorithm~3]{RNW18} as an AL subproblem solver for
the AL method proposed in Section \ref{sec:AL-method} from theoretical complexity perspective.

The following theorem states the iteration and operation complexity of Algorithm~\ref{alg:NCG}, whose proof is deferred to Section \ref{sec:pf-NCG}.

\begin{theorem}
\label{thm:NCG-iter-oper-cmplxity}
Suppose that Assumption~\ref{asp:NCG-cmplxity} holds. Let 
\begin{equation}
T_1:=\left\lceil\frac{\Fh-\Fl}{\min\{c_{\rm sol},c_{\rm nc}\}}\max\{\epsilon_g^{-2}\epsilon_H,\epsilon_H^{-3}\}\right\rceil+\left\lceil\frac{\Fh-\Fl}{c_{\rm nc}}\epsilon_H^{-3}\right\rceil+1, \
T_2:=\left\lceil\frac{\Fh-\Fl}{c_{\rm nc}}\epsilon_H^{-3}\right\rceil+1,\label{T1}
\end{equation}
where $\Fh=F(u^0)$, $\Fl$ is given in \eqref{lwbd-Hgupbd}, and 
\beqa
&&c_{\rm sol} :=\eta\min\left\{\left[\frac{4}{4+\zeta+\sqrt{(4+\zeta)^2+8L^F_H}}\right]^2,\left[\frac{\min\{6(1-\eta),2\}\theta}{L^F_H}\right]^2\right\}, \label{csol} \\
&& c_{\rm nc} := \frac{\eta}{16} \min\left\{1,\left[\frac{\min\{3(1-\eta),1\}\theta}{L_H^F}\right]^2\right\}.   \label{cnc}
\eeqa
 Then the following statements hold.
\begin{enumerate}[{\rm (i)}]
\item The total number of calls of Algorithm~\ref{pro:meo} in Algorithm~\ref{alg:NCG} is at most $T_2$.	
\item The total number of calls of Algorithm~\ref{alg:capped-CG} in Algorithm~\ref{alg:NCG} is at most $T_1$.
\item {\rm ({\bf iteration complexity})} Algorithm~\ref{alg:NCG} terminates in at most $T_1+T_2$ iterations with 
\begin{equation}\label{NCG-iter} 
T_1+T_2=\cO((\Fh-\Fl)(L_H^F)^2\max\{\epsilon_g^{-2}\epsilon_H,\epsilon_H^{-3}\}).
\end{equation}
Also, its output $x^t$ satisfies $\|\nabla F(x^t)\|\le \epsilon_g$ deterministically and
$\lambda_{\min}(\nabla^2F(x^t))$ $\ge-\epsilon_H$ with probability at least $1-\delta$ for
some $0 \le t \le T_1+T_2$.
\item {\rm ({\bf operation complexity})} Algorithm~\ref{alg:NCG} requires at most 
\[
\widetilde{\cO}((\Fh-\Fl)(L_H^F)^2\max\{\epsilon_g^{-2}\epsilon_H,\epsilon_H^{-3}\}\min\{n,(U_H^F/\epsilon_H)^{1/2}\})
\]
matrix-vector products, where $U_H^F$ is given in \eqref{lwbd-Hgupbd}.
\end{enumerate}
\end{theorem}
	
\section{A Newton-CG based AL method for problem~\eqref{model:equa-cnstr}}
\label{sec:AL-method}

In this section we propose a Newton-CG based AL method for finding a stochastic $(\epsilon_1,\epsilon_2)$-SOSP of problem \eqref{model:equa-cnstr} for any prescribed tolerances $\epsilon_1,\epsilon_2\in(0,1)$. Before proceeding, we make some additional assumptions on problem~\eqref{model:equa-cnstr}.

\begin{assumption}\label{asp:lowbd-knownfeas}		
\begin{enumerate}[{\rm (a)}]
\item An $\epsilon_1/2$-approximately feasible point $z_{\epsilon_1}$ of problem~\eqref{model:equa-cnstr}, namely satisfying $\|c(z_{\epsilon_1})\|\le\epsilon_1/2$, is known.	
\item There exist constants $\fh$, $\fl$ and $\gamma>0$, independent of $\epsilon_1$ and $\epsilon_2$,  such that
\begin{eqnarray}
&& f(z_{\epsilon_1})\le\fh, \label{hbd}\\
&& f(x)+\gamma\|c(x)\|^2/2\ge \fl,\quad \forall x\in\bR^n, \label{lbd}
\end{eqnarray}
 where $z_{\epsilon_1}$ is given in (a).
\item 
There exist some $\bdelta_f,\bdelta_c>0$ such that the set
\begin{equation}\label{nearly-feas-level-set}
\cS(\bdelta_f,\bdelta_c):=\{x:f(x)\le \fh+\bdelta_f,\ \|c(x)\|\le1+\bdelta_c\}
\end{equation}
is compact with $\fh$ given above. Also, $\nabla^2 f$ and $\nabla^2 c_i,\ i=1,2,\ldots,m$, are Lipschitz continuous in a convex open neighborhood, denoted by $\Omega(\bdelta_f,\bdelta_c)$, of $\cS(\bdelta_f,\bdelta_c)$.
\end{enumerate}
\end{assumption}

We now make some remarks on Assumption \ref{asp:lowbd-knownfeas}.

\begin{remark}
\begin{itemize}
 \item[(i)] A very similar assumption as Assumption~\ref{asp:lowbd-knownfeas}(a) was considered in \cite{CGLY17,GY19,LZ12,XW19}. By imposing Assumption~\ref{asp:lowbd-knownfeas}(a), we restrict our study on problem \eqref{model:equa-cnstr} for which an  $\epsilon_1/2$-approximately feasible point $z_{\epsilon_1}$ can be found by an inexpensive procedure. One example of such problem instances arises when there exists $v^0$ such that $\{x:\|c(x)\|\le \|c(v^0)\|\}$ is compact, $\nabla^2 c_i$, $1\le i\le m$, is Lipschitz continuous on a convex neighborhood of this set, and the LICQ holds on this set. Indeed, for this instance, a point $z_{\epsilon_1}$ satisfying $\|c(z_{\epsilon_1})\|\le\epsilon_1/2$ can be computed by applying our Newton-CG method (Algorithm~\ref{alg:NCG}) to the problem $\min_{x\in\bR^n}\|c(x)\|^2$. As seen from Theorem~\ref{thm:NCG-iter-oper-cmplxity}, the resulting iteration and operation complexity of Algorithm~\ref{alg:NCG} for finding such $z_{\epsilon_1}$ are respectively $\cO(\epsilon_1^{-3/2})$ and $\widetilde{\cO}(\epsilon_1^{-3/2}\min\{n,\epsilon_1^{-1/4}\})$, which are negligible compared with those of our AL method (see Theorems~\ref{thm:total-iter-cmplxity} and \ref{thm:total-iter-cmplxity2} below). As another example, when the standard error bound condition $\|c(x)\|^2=\cO(\|\nabla (\|c(x)\|^2)\|^\nu)$ holds on a level set of $\|c(x)\|$ for some $\nu>0$, one can find the above $z_{\epsilon_1}$ by applying a gradient method to the problem $\min_{x\in\bR^n}\|c(x)\|^2$ (e.g., see \cite{lu2022single,S19iAL}). In addition, the Newton-CG based AL method (Algorithm \ref{alg:2nd-order-AL-nonconvex}) proposed below is a second-order method with the aim to find a second-order stationary point. It is more expensive than a first-order method in general. To make best use of such an AL method in practice,  it is natural to run a first-order method in advance to obtain an $\epsilon_1/2$-first-order stationary point $z_{\epsilon_1}$
and then run the AL method using $z_{\epsilon_1}$ as an $\epsilon_1/2$-approximately feasible point. Therefore, Assumption~\ref{asp:lowbd-knownfeas}(a) is met in practice, provided that an $\epsilon_1/2$-first-order stationary point of \eqref{model:equa-cnstr} can be found by a first-order method.
	
\item[(ii)]
Assumption \ref{asp:lowbd-knownfeas}(b) is mild. In particular, the assumption in \eqref{hbd} holds if $f(x)\le\fh$ holds for all $x$ with $\|c(x)\|\le 1$, which is imposed in \cite[Assumption~3]{XW19}. It also holds if problem~\eqref{model:equa-cnstr} has a known feasible point, which is often imposed for designing AL methods for nonconvex constrained optimization (e.g., see \cite{LZ12,CGLY17,LL18,GY19}). Besides, the assumption in \eqref{lbd} implies that the quadratic penalty function is bounded below when the associated penalty parameter is sufficiently large, which is typically used in the study of quadratic penalty and AL methods for solving problem \eqref{model:equa-cnstr} (e.g., see \cite{HHZ17prox,GY19,XW19,KMM19}).
Clearly, when $\inf_{x\in {\mathbb{R}}^n}f(x) > -\infty$, one can see that \eqref{lbd} holds for any $\gamma > 0$. In general,  one possible approach to identifying $\gamma$ is to apply the techniques on infeasibility detection developed in the literature (e.g.,  \cite{BCN10if,BCW14,AT19})
to check the infeasibility of the level set $\{x: f(x)+\gamma\|c(x)\|^2/2\le \tilde{f}_{\low}\}$ for some sufficiently small $\tilde{f}_{\low}$. Note that this level set being infeasible for some $\tilde{f}_{\low}$ implies that \eqref{lbd} holds for the given $\gamma$ and $f_{\low}=\tilde{f}_{\low}$.

\item[(iii)]  Assumption \ref{asp:lowbd-knownfeas}(c) is not too restrictive. Indeed, the set $\cS(\delta_f,\delta_c)$ is compact if $f$ or $f(\cdot)+\gamma\|c(\cdot)\|^2/2$ is level-bounded. The latter level-boundedness assumption is commonly imposed for studying AL methods (e.g., see \cite{GY19,XW19}), which is stronger than our assumption.

\end{itemize}
\end{remark}

We next propose a Newton-CG based AL method in Algorithm \ref{alg:2nd-order-AL-nonconvex} for finding a stochastic $(\epsilon_1,\epsilon_2)$-SOSP of problem~\eqref{model:equa-cnstr} under Assumption \ref{asp:lowbd-knownfeas}. Instead of solving \eqref{model:equa-cnstr} directly, this method solves the perturbed problem:
\begin{equation}\label{model:equa-cnstr-pert}
\min_{x\in \bR^n}\ f(x)\quad \st\ \tilde{c}(x):=c(x)-c(z_{\epsilon_1})=0,
\end{equation}
where $z_{\epsilon_1}$ is given in Assumption~\ref{asp:lowbd-knownfeas}(a). Specifically, at the $k$th iteration, this method applies the Newton-CG method (Algorithm \ref{alg:NCG}) to find an approximate stochastic SOSP $x^{k+1}$  of the AL subproblem associated with \eqref{model:equa-cnstr-pert}:
\begin{equation}\label{tL-tc}
\min_{x\in\bR^n} \big\{\widetilde{\cL}(x,\lambda^k,\rho_k):= f(x)+(\lambda^k)^T\tilde{c}(x)+\rho_k\|\tilde{c}(x)\|^2/2\big\}
\end{equation}
such that $\widetilde{\cL}(x^{k+1},\lambda^k;\rho_k)$ is below a threshold (see \eqref{algstop:1st-order} and \eqref{algstop:2nd-order}), where $\lambda^k$ is a truncated Lagrangian multiplier, i.e., the one that results from projecting the standard multiplier  estimate $\tilde{\lambda}^k$ onto an Euclidean ball (see step 6 of Algorithm \ref{alg:2nd-order-AL-nonconvex}). The standard multiplier estimate $\tilde{\lambda}^{k+1}$ is then updated by the classical scheme described in step 4 of Algorithm \ref{alg:2nd-order-AL-nonconvex}. Finally, the penalty parameter $\rho_{k+1}$ is
adaptively updated based on the improvement on constraint violation (see step~\ref{altstep:penalty} of Algorithm~\ref{alg:2nd-order-AL-nonconvex}). Such a practical update scheme is often adopted in the literature (e.g., see \cite{B97,ABM08,CGLY17}).

We would like to point out that the truncated Lagrangian multiplier sequence $\{\lambda^k\}$ is used in the AL subproblems of  Algorithm \ref{alg:2nd-order-AL-nonconvex} and  is bounded, while the standard Lagrangian multiplier sequence $\{\tilde{\lambda}^k\}$ is used in those of the classical AL methods and can be unbounded. Therefore, Algorithm \ref{alg:2nd-order-AL-nonconvex} can be viewed as a safeguarded AL method. Truncated Lagrangian multipliers have been used in the literature for designing some AL methods  \cite{ABM08,BM14,KS17example,BM20}, and will play a crucial role in the subsequent complexity analysis of Algorithm \ref{alg:2nd-order-AL-nonconvex}.

\begin{algorithm}[h]
\caption{A Newton-CG based AL method for problem
\eqref{model:equa-cnstr}}
\label{alg:2nd-order-AL-nonconvex}
{\small
Let $\gamma$ be given in Assumption \ref{asp:lowbd-knownfeas}.\\
\noindent\textbf{Input}: $\epsilon_1,\epsilon_2\in(0,1)$, $\Lambda>0$, $x^0\in\bR^n$, $\lambda^0 \in \mathcal{B}_{\Lambda}$, $\rho_0>2\gamma$, $\alpha\in(0,1)$, ${r}>1$, $\delta\in(0,1)$, and  $z_{\epsilon_1}$ given in Assumption~\ref{asp:lowbd-knownfeas}.
\begin{algorithmic}[1]
\State Set $k=0$.
\State Set $\tau_k^g=\max\{\epsilon_1, {r}^{k\log\epsilon_1/\log 2}\}$ and $\tau_k^H=\max\{\epsilon_2, {r}^{k\log\epsilon_2/\log 2}\}$.
\State  Call Algorithm \ref{alg:NCG} with $\epsilon_g=\tau_k^g$, $\epsilon_H=\tau_k^H$ and $u^0=x^{k}_{\init}$ to find an approximate solution $x^{k+1}$ to $\min_{x\in\bR^n}\widetilde{\cL}(x,\lambda^k;\rho_k)$ such that 
\begin{eqnarray}
&&\widetilde{\cL}(x^{k+1},\lambda^k;\rho_k)\le f(z_{\epsilon_1}),\ \|\nabla_x \widetilde{\cL}(x^{k+1},\lambda^k;\rho_k)\|\le\tau_k^g,\label{algstop:1st-order}\\
&&\lambda_{\min}(\nabla^2_{xx}\widetilde{\cL}(x^{k+1},\lambda^k;\rho_k))\ge-\tau_k^H\  \text{with probability at least } 1-\delta, \label{algstop:2nd-order}
\end{eqnarray}
where
\begin{equation}\label{def:initial-iterate-subprob}
x^k_{\init}=\left\{\begin{array}{ll}
z_{\epsilon_1}&\text{if }\ \widetilde{\cL}(x^k,\lambda^k;\rho_k)> f(z_{\epsilon_1}),\\
x^{k}&\text{otherwise},
\end{array}\right.\quad\text{for } k\ge 0.
\end{equation}\label{algstp:ALsubpb}
\State Set $\tilde{\lambda}^{k+1}=\lambda^k+\rho_k\tilde{c}(x^{k+1})$.
\State If $\tau_k^g\le\epsilon_1$, $\tau_k^H\le\epsilon_2$ and $\|c(x^{k+1})\|\le\epsilon_1$, then output $(x^{k+1},\tilde{\lambda}^{k+1})$ and terminate.\label{algstep:stop}
\State Set $\lambda^{k+1}=\Pi_{\mathcal{B}_\Lambda}(\tilde{\lambda}^{k+1})$.\label{algstep:proj-multiplier}
\State If $k=0$ or $\|\tilde{c}(x^{k+1})\|>\alpha\|\tilde{c}(x^k)\|$, set $\rho_{k+1}={r}\rho_k$. Otherwise, set $\rho_{k+1}=\rho_k$.\label{altstep:penalty}
\State Set $k\leftarrow k+1$, and go to step 2.
\end{algorithmic}
}
\end{algorithm}	

\begin{remark}
\begin{itemize}
\item[(i)] Notice that the starting point $x_{\init}^0$ of Algorithm \ref{alg:2nd-order-AL-nonconvex} can be different from $z_{\epsilon_1}$ and it may be rather infeasible, though $z_{\epsilon_1}$ is a nearly feasible point of \eqref{model:equa-cnstr}.  Besides, $z_{\epsilon_1}$ is 
used to ensure convergence of Algorithm \ref{alg:2nd-order-AL-nonconvex}. Specifically, if the algorithm runs into a ``poorly infeasible point'' $x^k$, namely satisfying $\widetilde{\cL}(x^k,\lambda^k;\rho_k)> f(z_{\epsilon_1})$, it will be superseded by $z_{\epsilon_1}$  (see \eqref{def:initial-iterate-subprob}), which prevents the iterates $\{x^k\}$ from converging to an infeasible point. Yet, $x^k$ may be rather infeasible when $k$ is not large. Thus,  Algorithm \ref{alg:2nd-order-AL-nonconvex} substantially differs from a funneling or two-phase type algorithm, in which a nearly feasible point is found in Phase 1, and then approximate stationarity is sought while near feasibility is maintained throughout Phase 2 (e.g., see \cite{BGMST16,bueno2020complexity,cartis2013evaluation,cartis2014complexity,
cartis2015evaluation,cartis2019evaluation,CGT19eq,curtis2018complexity}).
\item[(ii)]
The choice of $\rho_0$ in Algorithm \ref{alg:2nd-order-AL-nonconvex} is mainly for the simplicity of complexity analysis. Yet, it may be overly large and lead to highly ill-conditioned AL subproblems in practice. To make Algorithm \ref{alg:2nd-order-AL-nonconvex} practically more efficient, one can possibly modify it by choosing a relatively small initial penalty parameter, then solving the subsequent AL subproblems by a first-order method until  an $\epsilon_1$-first-order stationary point $\hx$ of \eqref{model:equa-cnstr} along with a Lagrangian multiplier $\hlambda$ is found,  and finally performing the steps described in Algorithm \ref{alg:2nd-order-AL-nonconvex} but with $x^0=\hat{x}$ and $\lambda^0=\Pi_{\mathcal{B}_\Lambda}(\hlambda)$.
\end{itemize}
\end{remark}

Before analyzing the complexity of Algorithm \ref{alg:2nd-order-AL-nonconvex}, we first argue that it is well-defined if $\rho_0$ is suitably chosen.
Specifically, we will show that when $\rho_0$ is sufficiently large, one can apply the Newton-CG method (Algorithm \ref{alg:NCG}) to the AL subproblem $\min_{x\in \bR^n} \widetilde{\cL}(x,\lambda^k;\rho_k)$ with
$x^k_{\init}$ as the initial point to find an $x^{k+1}$ satisfying \eqref{algstop:1st-order} and \eqref{algstop:2nd-order}.
To this end, we start by noting from \eqref{hbd}, \eqref{model:equa-cnstr-pert}, \eqref{tL-tc} and \eqref{def:initial-iterate-subprob} that
\beq \label{L-xinit}
\widetilde{\cL}(x^{k}_{\init},\lambda^k;\rho_k) \le \max\{\widetilde{\cL}(z_{\epsilon_1},\lambda^k;\rho_k),f(z_{\epsilon_1})\}=f(z_{\epsilon_1})\le \fh.
\eeq
Based on the above observation, we show in the next lemma that when $\rho_0$ is sufficiently large, $\widetilde{\cL}(\cdot,\lambda^k;\rho_k)$ is bounded below and its certain level set is bounded, whose proof is deferred to Section \ref{sec:pf-AL}.

\begin{lemma}
\label{lem:level-set-augmented-lagrangian-func}
Suppose that Assumption~\ref{asp:lowbd-knownfeas} holds. Let $(\lambda^k, \rho_k)$ be generated at the $k$th iteration of Algorithm~\ref{alg:2nd-order-AL-nonconvex} for some $k \ge 0$, and $\cS(\bdelta_f,\bdelta_c)$ and $x_{\init}^k$ be defined in \eqref{nearly-feas-level-set} and \eqref{def:initial-iterate-subprob}, respectively, and let $\fh$, $\fl$, $\bdelta_f$ and $\bdelta_c$ be given in Assumption~\ref{asp:lowbd-knownfeas}. Suppose that $\rho_0$ is sufficiently large such that $\delta_{f,1} \le \bdelta_f$ and  $\delta_{c,1} \le \bdelta_c$, where 
\begin{equation} 
\delta_{f,1}:=\Lambda^2/(2\rho_0)\quad and\quad \delta_{c,1}:=\sqrt{\frac{2(\fh-\fl+\gamma)}{\rho_0-2\gamma}+\frac{\Lambda^2}{(\rho_0-2\gamma)^2}}+\frac{\Lambda}{\rho_0-2\gamma}\label{def:delta0c-rhobar1xxx}.
\end{equation}
Then the following statements hold.
\begin{enumerate}[{\rm (i)}]
\item $\{x:\widetilde{\cL}(x,\lambda^k;\rho_k)\le \widetilde{\cL}(x^{k}_{\init},\lambda^k;\rho_k)\}\subseteq \cS(\bdelta_f,\bdelta_c)$.
\item $\inf_{x\in\bR^n} \widetilde{\cL}(x,\lambda^k;\rho_k) \ge\fl - \gamma-\Lambda\bdelta_c$.
\end{enumerate}
\end{lemma}
	
Using Lemma \ref{lem:level-set-augmented-lagrangian-func}, we can verify that the Newton-CG method (Algorithm \ref{alg:NCG}), starting with $u^0=x^{k}_{\init}$, is capable of finding an approximate solution $x^{k+1}$ of the AL subproblem $\min_{x\in\bR^n} \widetilde{\cL}(x,\lambda^k;\rho_k)$ satisfying \eqref{algstop:1st-order} and \eqref{algstop:2nd-order}. Indeed,
let $F(\cdot)=\widetilde{\cL}(\cdot,\lambda^k;\rho_k)$ and $u^0=x^{k}_{\init}$. By these and Lemma \ref{lem:level-set-augmented-lagrangian-func}, one can see that $\{x:F(x) \le F(u^0)\} \subseteq \cS(\bdelta_f,\bdelta_c)$. It then follows from this and Assumption \ref{asp:lowbd-knownfeas}(c) that the level set $\{x : F(x) \le F(u^0)\}$ is compact  and $\nabla^2 F$ is Lipschitz continuous on a convex open neighborhood of $\{x:F(x) \le F(u^0)\}$. Thus, such $F$ and $u^0$ satisfy Assumption \ref{asp:NCG-cmplxity}. Based on this and the discussion in Section \ref{sec:sbpb-solver}, one can conclude that Algorithm~\ref{alg:NCG}, starting with $u^0=x^{k}_{\init}$, is applicable to the AL subproblem $\min_{x\in\bR^n} \widetilde{\cL}(x,\lambda^k;\rho_k)$. Moreover, it follows from Theorem \ref{thm:NCG-iter-oper-cmplxity} that this algorithm with $(\epsilon_g,\epsilon_H)=(\tau_k^g,\tau_k^H)$ can produce a point $x^{k+1}$ satisfying \eqref{algstop:2nd-order} and also the second relation in \eqref{algstop:1st-order}. In addition, since this algorithm
is descent and its starting point is $x^{k}_{\init}$, its output $x^{k+1}$ must satisfy $\widetilde{\cL}(x^{k+1},\lambda^k;\rho_k) \le \widetilde{\cL}(x^{k}_{\init},\lambda^k;\rho_k)$, which along with \eqref{L-xinit} implies that $\widetilde{\cL}(x^{k+1},\lambda^k;\rho_k) \le f(z_{\epsilon_1})$ and thus $x^{k+1}$ also satisfies the first relation in \eqref{algstop:1st-order}.

The above discussion leads to the following conclusion concerning the {\em well-definedness of Algorithm~\ref{alg:2nd-order-AL-nonconvex}}.

\begin{theorem}
\label{subprob-solver}
Under the same settings as in Lemma \ref{lem:level-set-augmented-lagrangian-func}, the Newton-CG method (Algorithm \ref{alg:NCG}) applied to the AL subproblem $\min_{x\in\bR^n} \widetilde{\cL}(x,\lambda^k;\rho_k)$ with $u^0=x^{k}_{\init}$ finds a point $x^{k+1}$ satisfying \eqref{algstop:1st-order} and \eqref{algstop:2nd-order}.
\end{theorem}

The following theorem characterizes the {\em output of Algorithm \ref{alg:2nd-order-AL-nonconvex}}.
Its proof is deferred to Section \ref{sec:pf-AL}.

\begin{theorem}
\label{thm:output-alg1}
Suppose that Assumption~\ref{asp:lowbd-knownfeas} holds and that $\rho_0$ is sufficiently large such that $\delta_{f,1} \le \bdelta_f$ and  $\delta_{c,1} \le \bdelta_c$, where
$\delta_{f,1}$ and $\delta_{c,1}$ are defined in \eqref{def:delta0c-rhobar1xxx}.
If Algorithm \ref{alg:2nd-order-AL-nonconvex} terminates at some iteration $k$, then $x^{k+1}$ is a deterministic $\epsilon_1$-FOSP of problem \eqref{model:equa-cnstr}, and moreover, it is an $(\epsilon_1,\epsilon_2)$-SOSP of \eqref{model:equa-cnstr} with probability at least $1-\delta$.
\end{theorem}
	
\begin{remark}\label{only-1st}
As seen from this theorem, the output of Algorithm \ref{alg:2nd-order-AL-nonconvex} is a stochastic $(\epsilon_1, \epsilon_2)$-SOSP of problem~\eqref{model:equa-cnstr}. Nevertheless, one can easily modify Algorithm \ref{alg:2nd-order-AL-nonconvex} to seek some other approximate solutions. For example, if one is only interested in finding an $\epsilon_1$-FOSP of \eqref{model:equa-cnstr}, one can remove the condition \eqref{algstop:2nd-order} from Algorithm \ref{alg:2nd-order-AL-nonconvex}. In addition, if one aims to find a deterministic $(\epsilon_1, \epsilon_2)$-SOSP of \eqref{model:equa-cnstr}, one can replace the condition \eqref{algstop:2nd-order} and Algorithm \ref{alg:NCG} by $\lambda_{\min}(\nabla^2_{xx}\widetilde{\cL}(x^{k+1},\lambda^k;\rho_k))\ge-\tau_k^H$ and a deterministic counterpart, respectively. The purpose of imposing high probability in the condition \eqref{algstop:2nd-order} is to enable us to derive  operation complexity of
Algorithm \ref{alg:2nd-order-AL-nonconvex} measured by the number of matrix-vector products.
\end{remark}

In the rest of this section, we study the worst-case complexity of Algorithm \ref{alg:2nd-order-AL-nonconvex}. Since our method has two nested loops, particularly,  outer loops executed by the AL method and inner loops executed by the Newton-CG method for solving the AL subproblems, we consider the following measures of complexity for Algorithm \ref{alg:2nd-order-AL-nonconvex}.
\begin{itemize}
\item \emph{Outer iteration complexity}, which measures the number of outer iterations of Algorithm \ref{alg:2nd-order-AL-nonconvex};
\item  \emph{Total inner iteration complexity}, which measures the total number of iterations of the Newton-CG method that are performed in Algorithm \ref{alg:2nd-order-AL-nonconvex};
\item \emph{Operation complexity}, which measures the total number of matrix-vector products involving the Hessian of the augmented Lagrangian function that are evaluated in Algorithm \ref{alg:2nd-order-AL-nonconvex}.
\end{itemize}

\subsection{Outer iteration complexity of Algorithm \ref{alg:2nd-order-AL-nonconvex}}

In this subsection we establish outer iteration complexity of Algorithm~\ref{alg:2nd-order-AL-nonconvex}.
For notational convenience, we rewrite $(\tau_k^g,\tau_k^H)$ arising in Algorithm \ref{alg:2nd-order-AL-nonconvex} as
\begin{equation}\label{omega-tolerance}
(\tau_k^g,\tau_k^H)=(\max\{\epsilon_1,\omega_1^k\},\max\{\epsilon_2,\omega_2^k\})\ \mbox{ with }\ (\omega_1,\omega_2):=({r}^{\log\epsilon_1/\log 2},{r}^{\log\epsilon_2/\log 2}),
\end{equation}
where $\epsilon_1$, $\epsilon_2$ and $r$ are the input parameters of Algorithm \ref{alg:2nd-order-AL-nonconvex}. Since $r>1$ and $\epsilon_1,\epsilon_2\in(0,1)$, it is not hard to verify that $\omega_1,\omega_2\in(0,1)$.
Also, we introduce the following quantity that will be used frequently later:
\begin{equation}
K_{\epsilon_1}:= \left\lceil \min\{k\ge 0:\omega_1^{k}\le\epsilon_1\} \right\rceil=\left\lceil\log\epsilon_1/\log\omega_1\right\rceil. \label{T-epsilon-g}
\end{equation}
In view of \eqref{omega-tolerance}, \eqref{T-epsilon-g} and the fact that
\begin{equation}\label{epsw1-epsw2}
\log\epsilon_1/\log\omega_1=\log\epsilon_2/\log\omega_2=\log2/\log r,
\end{equation}
we see that $(\tau_k^g,\tau_k^H)=(\epsilon_1,\epsilon_2)$ for all $k\ge K_{\epsilon_1}$. This along with the termination criterion of Algorithm \ref{alg:2nd-order-AL-nonconvex} implies that it runs for at least $K_{\epsilon_1}$ iterations and  terminates once $\|c(x^{k+1})\|\le\epsilon_1$ for some $k\ge K_{\epsilon_1}$. As a result, to establish outer iteration complexity of Algorithm \ref{alg:2nd-order-AL-nonconvex}, it suffices to bound such  $k$. The resulting outer iteration complexity of Algorithm \ref{alg:2nd-order-AL-nonconvex} is presented below, whose proof is deferred to Section \ref{sec:pf-AL}.

\begin{theorem}
\label{thm:out-itr-cmplxity-1}
Suppose that Assumption~\ref{asp:lowbd-knownfeas} holds and that $\rho_0$ is sufficiently large such that $\delta_{f,1} \le \bdelta_f$ and  $\delta_{c,1} \le \bdelta_c$, where
$\delta_{f,1}$ and $\delta_{c,1}$ are defined in \eqref{def:delta0c-rhobar1xxx}. Let
\begin{eqnarray}
&&{\rho}_{\epsilon_1}:=\max\left\{8(\fh-\fl+\gamma)\epsilon_1^{-2}+4\Lambda\epsilon_1^{-1}+2\gamma,2\rho_0\right\},\label{def:delta0c-rhobar1}\\
&&\overline{K}_{\epsilon_1}:=\inf\{k\ge K_{\epsilon_1}: \|c(x^{k+1})\|\le\epsilon_1\},\label{number-outer-iteration}
\end{eqnarray}
where $K_{\epsilon_1}$ is defined in \eqref{T-epsilon-g},  and $\gamma$, $\fh$ and $\fl$ are given in Assumption \ref{asp:lowbd-knownfeas}. Then $\overline{K}_{\epsilon_1}$ is finite, and Algorithm \ref{alg:2nd-order-AL-nonconvex} terminates at iteration $\overline{K}_{\epsilon_1}$ with
\begin{equation}\label{outer-iteration-cmplxty}
\overline{K}_{\epsilon_1}\le\left(\frac{\log({\rho}_{\epsilon_1}\rho_0^{-1})}{\log{r}}+1\right)\left(\left|\frac{\log(\epsilon_1(2\delta_{c,1})^{-1})}{\log \alpha}\right|+2\right)+1.
\end{equation}
Moreover, $\rho_k\le{r}{\rho}_{\epsilon_1}$ holds for $0 \le k \le \overline{K}_{\epsilon_1}$
\end{theorem}

\begin{remark}[{{\bf Upper bounds for $\overline{K}_{\epsilon_1}$ and $\{\rho_k\}$}}]\label{order-out-itera-penalty}
As observed from Theorem \ref{thm:out-itr-cmplxity-1}, the number of outer iterations of Algorithm \ref{alg:2nd-order-AL-nonconvex} for finding a stochastic  $(\epsilon_1,
\epsilon_2)$-SOSP of problem \eqref{model:equa-cnstr} is $\overline{K}_{\epsilon_1}+1$, which is at most of $\cO(|\log\epsilon_1|^2)$. In addition, the penalty parameters $\{\rho_k\}$ generated in this algorithm are at most of $\cO(\epsilon_1^{-2})$.
\end{remark}
	
\subsection{Total inner iteration and operation complexity of Algorithm \ref{alg:2nd-order-AL-nonconvex}}\label{subsec:total-complexity}
	

We present the total inner iteration and operation complexity of Algorithm \ref{alg:2nd-order-AL-nonconvex} for finding a stochastic ($\epsilon_1,
\epsilon_2$)-SOSP of \eqref{model:equa-cnstr}, whose proof is deferred to Section \ref{sec:pf-AL}.

\begin{theorem}
\label{thm:total-iter-cmplxity}
Suppose that Assumption \ref{asp:lowbd-knownfeas} holds and that $\rho_0$ is sufficiently large such that $\delta_{f,1} \le \bdelta_f$ and  $\delta_{c,1} \le \bdelta_c$, where $\delta_{f,1}$ and $\delta_{c,1}$ are defined in \eqref{def:delta0c-rhobar1xxx}. Then the following statements hold.
\begin{enumerate}[{\rm (i)}]
\item 
The total number of iterations of Algorithm~\ref{alg:NCG} performed in Algorithm \ref{alg:2nd-order-AL-nonconvex} is at most $\widetilde{\cO}(\epsilon_1^{-4}\max\{\epsilon_1^{-2}\epsilon_2,\epsilon_2^{-3}\})$. If $c$ is further assumed to be affine, then it is at most $\widetilde{\cO}(\max\{\epsilon_1^{-2}\epsilon_2,\epsilon_2^{-3}\})$.
\item The total number of matrix-vector products performed  by Algorithm~\ref{alg:NCG} in  Algorithm \ref{alg:2nd-order-AL-nonconvex} is at most $\widetilde{\cO}(\epsilon_1^{-4}\max\{\epsilon_1^{-2}\epsilon_2,\epsilon_2^{-3}\}\min\{n,\epsilon_1^{-1}\epsilon_2^{-1/2}\})$. If $c$ is further assumed to be affine, then it is at most $\widetilde{\cO}(\max\{\epsilon_1^{-2}\epsilon_2,\epsilon_2^{-3}\}\min\{n,\epsilon_1^{-1}\epsilon_2^{-1/2}\})$.
\end{enumerate}
\end{theorem}
	
\begin{remark}
\begin{enumerate}[{\rm (i)}]
\item Note that the above complexity results of Algorithm \ref{alg:2nd-order-AL-nonconvex} are established without assuming any constraint qualification (CQ). In contrast, similar complexity results are obtained in \cite{XW19} for a proximal AL method under a generalized LICQ condition. To the best of our knowledge, our work provides the first study on complexity for finding a stochastic SOSP of  \eqref{model:equa-cnstr} without CQ.
\item Letting $(\epsilon_1,\epsilon_2)=(\epsilon,\sqrt{\epsilon})$ for some $\epsilon\in(0,1)$, we see that Algorithm~\ref{alg:2nd-order-AL-nonconvex} achieves a total inner iteration complexity of $\widetilde{\cO}(\epsilon^{-11/2})$ and an operation complexity of $\widetilde{\cO}(\epsilon^{-11/2}\min\{n,\epsilon^{-5/4}\})$ for finding a stochastic $(\epsilon,\sqrt{\epsilon})$-SOSP of problem~\eqref{model:equa-cnstr} without constraint qualification.
\end{enumerate}
\end{remark}

\subsection{Enhanced complexity of Algorithm \ref{alg:2nd-order-AL-nonconvex} under constraint qualification}\label{sec:AL-modified}

In this subsection we study complexity of Algorithm \ref{alg:2nd-order-AL-nonconvex} under one additional assumption that a generalized linear independence constraint qualification (GLICQ) holds for problem \eqref{model:equa-cnstr}, which is introduced below. In particular, under GLICQ we will obtain an enhanced total inner iteration and operation complexity for Algorithm \ref{alg:2nd-order-AL-nonconvex}, which are significantly better than the ones in Theorem \ref{thm:total-iter-cmplxity} when problem \eqref{model:equa-cnstr} has nonlinear constraints. Moreover, when $(\epsilon_1,\epsilon_2)=(\epsilon,\sqrt{\epsilon})$ for some $\epsilon\in (0,1)$, our enhanced complexity bounds are also better than those obtained in \cite{XW19} for a proximal AL method. We now introduce the GLICQ assumption for problem \eqref{model:equa-cnstr}.

\begin{assumption}[{{\bf GLICQ}}]\label{asp:LICQ}
$\nabla c(x)$ has full column rank for all $x\in\cS(\bdelta_f,\bdelta_c)$, where $\cS(\bdelta_f,\bdelta_c)$ is as in \eqref{nearly-feas-level-set}.
\end{assumption}

\begin{remark}
A related yet different GLICQ is imposed in \cite[Assumption~2(ii)]{XW19} for problem \eqref{model:equa-cnstr}, which assumes that $\nabla c(x)$ has full column rank for all $x$ in a level set of $f(\cdot)+\gamma\|c(\cdot)\|^2/2$. It is not hard to verify that this assumption is generally stronger than the above GLICQ assumption.
\end{remark}

The following theorem shows that under Assumption \ref{asp:LICQ}, the total inner iteration and operation complexity results presented in Theorem \ref{thm:total-iter-cmplxity} can be significantly improved, whose proof is deferred to Section \ref{sec:pf-AL}.

\begin{theorem}
\label{thm:total-iter-cmplxity2}
Suppose that Assumptions \ref{asp:lowbd-knownfeas} and \ref{asp:LICQ} hold and that $\rho_0$ is sufficiently large such that $\delta_{f,1} \le \bdelta_f$ and  $\delta_{c,1} \le \bdelta_c$, where $\delta_{f,1}$ and $\delta_{c,1}$ are defined in \eqref{def:delta0c-rhobar1xxx}. Then the following statements hold.
\begin{enumerate}[{\rm (i)}]
\item 
The total number of iterations of Algorithm~\ref{alg:NCG} performed in Algorithm \ref{alg:2nd-order-AL-nonconvex} is at most $\widetilde{\cO}(\epsilon_1^{-2}\max\{\epsilon_1^{-2}\epsilon_2,\epsilon_2^{-3}\})$. If $c$ is further assumed to be affine, then it is at most $\widetilde{\cO}(\max\{\epsilon_1^{-2}\epsilon_2,\epsilon_2^{-3}\})$.
\item The total number of matrix-vector products performed  by Algorithm~\ref{alg:NCG} in  Algorithm \ref{alg:2nd-order-AL-nonconvex} is at most $\widetilde{\cO}(\epsilon_1^{-2}\max\{\epsilon_1^{-2}\epsilon_2,\epsilon_2^{-3}\}\min\{n,\epsilon_1^{-1/2}\epsilon_2^{-1/2}\})$. If $c$ is further assumed to be affine, then it is at most $\widetilde{\cO}(\max\{\epsilon_1^{-2}\epsilon_2,\epsilon_2^{-3}\}\min\{n,\epsilon_1^{-1/2}\epsilon_2^{-1/2}\})$.
\end{enumerate}
\end{theorem}

\begin{remark}
\begin{enumerate}[{\rm (i)}]
\item As seen from Theorem \ref{thm:total-iter-cmplxity2},  when problem \eqref{model:equa-cnstr} has nonlinear constraints, under GLICQ and some other suitable assumptions, Algorithm \ref{alg:2nd-order-AL-nonconvex} achieves significantly better complexity bounds than the ones in Theorem \ref{thm:total-iter-cmplxity} without constraint qualification.
\item Letting $(\epsilon_1,\epsilon_2)=(\epsilon,\sqrt{\epsilon})$ for some $\epsilon\in(0,1)$, we see that when problem \eqref{model:equa-cnstr} has nonlinear constraints, under GLICQ and some other suitable assumptions, Algorithm~\ref{alg:2nd-order-AL-nonconvex} achieves a total inner iteration complexity of $\widetilde{\cO}(\epsilon^{-7/2})$ and an operation complexity of $\widetilde{\cO}(\epsilon^{-7/2}\min\{n,\epsilon^{-3/4}\})$. They are vastly better than the total inner iteration complexity of $\widetilde{\cO}(\epsilon^{-11/2})$ and the operation complexity of $\widetilde{\cO}(\epsilon^{-11/2}\min\{n,\epsilon^{-3/4}\})$ that are achieved by a proximal AL method in \cite{XW19} for finding a stochastic $(\epsilon,
\sqrt{\epsilon})$-SOSP of \eqref{model:equa-cnstr} yet under a generally stronger GLICQ.
\end{enumerate}
\end{remark}

\section{Numerical results}\label{sec:nr}
We conduct some preliminary experiments to test the performance of our proposed methods (Algorithms~\ref{alg:NCG} and \ref{alg:2nd-order-AL-nonconvex}), and compare them with the Newton-CG method in \cite{RNW18} and the proximal AL method in \cite{XW19}, respectively. All the algorithms are coded in Matlab and all the computations are performed on a desktop with a 3.79 GHz AMD 3900XT 12-Core processor and 32 GB of RAM.

\subsection{Regularized robust regression} \label{rrr}

In this subsection we consider the regularized robust regression problem
\begin{equation}\label{uncnstr}
\min_{x\in\bR^n}\ \sum_{i=1}^m \phi(a_i^Tx - b_i) + \mu \|x\|_4^4,
\end{equation}
where $\phi(t) = t^2 /(1 + t^2)$, $\|x\|_p = (\sum_{i=1}^n |x_i|^p)^{1/p}$ for any $p \ge 1$, and $\mu>0$.

For each triple $(n,m,\mu)$, we randomly generate 10 instances of problem~\eqref{uncnstr}. In particular, we first randomly generate $a_i$, $1\le i\le m$, with all the entries independently chosen from the standard normal distribution. We then randomly generate $\bar b_i$ according to the standard normal distribution and set $b_i= 2m\bar{b}_i$ for $i=1,\ldots, m$.

Our aim is to find a $(10^{-5},10^{-5/2})$-SOSP of \eqref{uncnstr} for the above instances by Algorithm~\ref{alg:NCG} and the Newton-CG method in \cite{RNW18} and compare their performance. For a fair comparison, we use a minimum eigenvalue oracle that returns a deterministic output for them so that they both certainly output an approximate second-order stationary point. Specifically,  we use the Matlab subroutine \textsf{[v,$\lambda$] = eigs(H,1,'smallestreal')} as the minimum eigenvalue oracle to find the minimum eigenvalue $\lambda$ and its associated unit eigenvector $v$ of a real symmetric matrix $H$. Also, for both methods, we choose the all-ones vector as the initial point, and set $\theta = 0.8$, $\zeta = 0.5$, and $\eta = 0.2$.

The computational results of Algorithm~\ref{alg:NCG} and the Newton-CG method in \cite{RNW18} for the instances randomly generated above are presented in Table~\ref{table:num-uncnstr-1}. In detail, the value of $n$, $m$, and $\mu$  is listed in the first three columns, respectively. For each triple $(n,m,\mu)$, the average CPU time (in seconds), the average number of iterations, and the average final objective value over 10 random instances are given in the rest of the columns. One can observe that both methods output an approximate solution with a similar objective value, while our Algorithm~\ref{alg:NCG} substantially outperforms the Newton-CG method in \cite{RNW18} in terms of CPU time. This is consistent with our theoretical finding that Algorithm~\ref{alg:NCG} achieves a better iteration complexity than the Newton-CG method in \cite{RNW18} in terms of dependence on the Lipschitz constant of the Hessian for finding an approximate SOSP.

\begin{table}[t]
{\footnotesize
\centering
\begin{tabular}{ccc||ll||ll||ll}
\hline
& & &\multicolumn{2}{c||}{Objective value} & \multicolumn{2}{c||}{Iterations} & \multicolumn{2}{c}{CPU time (seconds)}\\
$n$ & $m$ & $\mu$ & Algorithm 1 & Newton-CG  & Algorithm 1 & Newton-CG & Algorithm 1 & Newton-CG\\ \hline
100 & 10 & 1 & 5.9 &    5.9 & 85.7 &  116.3 & 1.4 & 1.6 \\ 
100 & 50 & 1 &   45.9 &  45.9 &  82.6  & 158.2  & 1.0  &  2.7  \\ 
100 & 90 & 1 & 84.8 &  84.8 & 102.2 & 224.7 & 2.0  &  4.2  \\ 
500& 50 & 5 & 42.2 &  42.5 & 173.1 & 344.7 & 44.2  & 72.2  \\ 
500 & 250 & 5 & 243.0  & 242.9 & 145.5 & 362.4 & 41.9 &  95.0  \\ 
500 & 450 & 5 & 442.2 & 442.2 & 163.7 & 425.2 & 47.6 & 138.3 \\ 
1000 & 100 & 10 & 90.1 &  90.4  & 162.5 & 361.0 & 110.8 & 259.0 \\ 
1000 & 500 & 10 & 491.1 & 491.2 & 158.3 & 475.4 & 129.1 & 558.4 \\ 
1000 & 900 & 10 &  891.1 & 891.1 & 193.5 & 300.7 & 187.0 & 298.5 \\\hline 
\end{tabular}
\caption{Numerical results for problem~\eqref{uncnstr}}
\label{table:num-uncnstr-1}
}
\end{table}

\subsection{Spherically constrained regularized robust regression}
In subsection we consider the spherically constrained regularized robust regression problem
\begin{equation}\label{cnstr-rlr}
\min_{x\in\bR^n}\ \sum_{i=1}^m \phi(a_i^Tx - b_i) + \mu \|x\|_4^4\quad \st\quad \|x\|_2^2 = 1,
\end{equation}
where $\phi(t) = t^2 /(1 + t^2)$, $\|x\|_p = (\sum_{i=1}^n |x_i|^p)^{1/p}$ for any $p \ge 1$, and $\mu>0$ is a tuning parameter. For each triple {$(n,m,\mu)$,  we randomly generate 10 instances of problem~\eqref{cnstr-rlr} in the same manner as described in Subsection \ref{rrr}.

Our aim is to find a $(10^{-4},10^{-2})$-SOSP of \eqref{cnstr-rlr} for the above instances by Algorithm~\ref{alg:2nd-order-AL-nonconvex} and the proximal AL method \cite[Algorithm~3]{XW19} and compare their performance. For a fair comparison, we use a minimum eigenvalue oracle that returns a deterministic output for them so that they both certainly output an approximate second-order stationary point. Specifically,  we use the Matlab subroutine \textsf{[v,$\lambda$] = eigs(H,1,'smallestreal')} as the minimum eigenvalue oracle to find the minimum eigenvalue $\lambda$ and its associated unit eigenvector $v$ of a real symmetric matrix $H$.  In addition, for both methods, we choose the initial point as $z^0=(1/\sqrt{n},\ldots,1/\sqrt{n})^T$, the initial Lagrangian multiplier as $\lambda^0=0$, and the other parameters as
\begin{itemize}
\item  $\Lambda=100$, $\rho_0=10$, $\alpha=0.25$, and $r=10$ for Algorithm~\ref{alg:2nd-order-AL-nonconvex};
\item $\eta=1$, $q=10$ and  $T_0=2$ for the proximal AL method (\cite{XW19}).
\end{itemize}

The computational results of Algorithm~\ref{alg:2nd-order-AL-nonconvex} and the proximal AL method in \cite{XW19} (abbreviated as Prox-AL) for solving problem~\eqref{cnstr-rlr} for the instances randomly generated above are presented in Table~\ref{table:AL-num}. In detail, the value of $n$, $m$, and $\mu$  is listed in the first three columns, respectively. For each triple $(n,m,\mu)$, the average CPU time (in seconds), the average total number of inner iterations, the average final objective value, and the average final feasibility violation over 10 random instances are given in the rest columns. One can observe that both methods output an approximate solution of similar quality in terms of objective value and feasibility violation, while our Algorithm~\ref{alg:2nd-order-AL-nonconvex} vastly outperforms the proximal AL method in \cite{XW19} in terms of CPU time. This corroborates our theoretical finding that Algorithm~\ref{alg:2nd-order-AL-nonconvex} achieves a significantly better operation complexity than the proximal AL method in \cite{XW19} for finding an approximate SOSP.


\begin{table}
{\footnotesize
\centering
\begin{tabular}{ccc||ll||ll||ll||ll}
\hline
& & &\multicolumn{2}{c||}{Objective value} & \multicolumn{2}{c||}{Feasibility violation ($\times 10^{-4}$)} & \multicolumn{2}{c||}{Total inner iterations} & \multicolumn{2}{c}{CPU time (seconds)}\\
$n$ & $m$ & $\mu$ & Algorithm 2 & Prox-AL  & Algorithm 2 & Prox-AL & Algorithm 2 & Prox-AL &Algorithm 2 & Prox-AL\\ \hline
100 & 10 & 1 & 7.1 & 7.1 & 0.18 & 0.27  & 40.9  & 97.3 & 0.73 & 2.2\\
100 & 50 & 1  & 46.6 & 46.6 & 0.21 & 0.30  & 37.0 & 86.3& 0.78 & 1.7\\
100 & 90 & 1 & 87.0 & 87.0 & 0.12& 0.40  & 39.5 & 68.6& 1.1 & 1.9\\
500 & 50 & 5 &44.4&44.4&0.40&0.68 &59.0&343.4 & 11.4&134.9\\
500& 250 & 5 &244.3&244.3&0.37&0.47&59.0&543.3& 11.7&178.2\\
500 & 450 & 5 &444.0&444.0&0.27&0.53&66.7&634.1&17.1&158.2\\
1000 & 100 & 10 &92.8&92.8&0.28&0.42&95.0&2054.6&46.3&1516.8\\
1000 & 500 & 10 &491.9&491.9&0.22&0.72&68.3&756.2& 39.5&558.6\\
1000 & 900 & 10 &893.4&893.4&0.19&0.37&81.8&1281.4& 57.7&1099.6\\\hline
\end{tabular}
\caption{Numerical results for problem~\eqref{cnstr-rlr}}
\label{table:AL-num}
}
\end{table}

\section{Proof of the main results}\label{sec:proof}
We provide proofs of our main results in Sections \ref{sec:sbpb-solver} and \ref{sec:AL-method}, including Theorem \ref{thm:NCG-iter-oper-cmplxity}, Lemma \ref{lem:level-set-augmented-lagrangian-func}, and Theorems \ref{thm:output-alg1},  \ref{thm:out-itr-cmplxity-1}, \ref{thm:total-iter-cmplxity} and \ref{thm:total-iter-cmplxity2}.
	
\subsection{Proof of the main results in Section \ref{sec:sbpb-solver}}
\label{sec:pf-NCG}

In this subsection we first establish several technical lemmas and then use them to prove Theorem \ref{thm:NCG-iter-oper-cmplxity}.

One can observe from Assumption~\ref{asp:NCG-cmplxity}(b) that for all $x$ and $y\in \Omega$,
\begin{eqnarray}
&&\ \ \ \|\nabla F(y)-\nabla F(x)-\nabla^2 F(x)(y-x)\|\le L_H^F\|y-x\|^2/2, \label{apx-nxt-grad}\\
&&\ \ \ F(y)\le F(x) + \nabla F(x)^T(y-x) + (y-x)^T\nabla^2 F(x) (y-x)/2 +  L_H^F\|y-x\|^3/6. \label{desc-ineq}
\end{eqnarray}

The next lemma provides useful properties of the output of Algorithm~\ref{alg:capped-CG}, whose proof is similar to the ones in \cite[Lemma~3]{RNW18} and \cite[Lemma~7]{NW19} and thus omitted here.

\begin{lemma}\label{lem:SOL-NC-ppty}
Suppose that Assumption~\ref{asp:NCG-cmplxity} holds and the direction $d^t$ results from  the output $d$ of Algorithm \ref{alg:capped-CG} with a type specified in d$\_$type  at some iteration $t$ of Algorithm \ref{alg:NCG}. Then the following statements hold.
\begin{enumerate}[{\rm (i)}]
\item If d$\_$type=SOL, then  $d^t$ satisfies
\beqa
&\epsilon_H\|d^t\|^2\le (d^t)^T\left(\nabla^2 F(x^t)+2\epsilon_H I\right)d^t,\label{SOL-ppty-1}\\
&\|d^t\|\le1.1 \epsilon_H^{-1}\|\nabla F(x^t)\|,\label{SOL-ppty-2}\\
&(d^t)^T\nabla F(x^t)=-(d^t)^T\left(\nabla^2F(x^t)+2\epsilon_HI\right)d^t,\label{SOL-ppty-3}\\
&\|(\nabla^2F(x^t)+2\epsilon_HI)d^t+\nabla F(x^t)\|\le\epsilon_H\zeta\|d^t\|/2.\label{SOL-ppty-4}
\eeqa
\item If d$\_$type=NC, then  $d^t$ satisfies $(d^t)^T\nabla F(x^t)\le0$ and
\begin{equation}\label{NC-ppty}
(d^t)^T\nabla^2F(x^t)d^t/\|d^t\|^2=-\|d^t\|\le-\epsilon_H.
\end{equation}
\end{enumerate}
\end{lemma}

The next lemma shows that when the search direction $d^t$ in Algorithm \ref{alg:NCG} is of type `SOL', the line search step results in a sufficient reduction on $F$.

\begin{lemma}\label{lem:sol}
Suppose that Assumption~\ref{asp:NCG-cmplxity} holds and the direction $d^t$ results from  the output $d$ of Algorithm~\ref{alg:capped-CG} with d$\_$type=SOL at some iteration $t$ of Algorithm~\ref{alg:NCG}. Let $U_g^F$ and $c_{\rm sol}$ be given in \eqref{lwbd-Hgupbd} and \eqref{csol}, respectively.  Then the following statements hold.
\begin{enumerate}[{\rm (i)}]
\item The step length $\alpha_t$ is well-defined, and moreover,
\begin{equation}\label{lwbd-ak-sol}
\alpha_t\ge\min\left\{1,\sqrt{\frac{\min\{6(1-\eta),2\}}{1.1 L_H^F U_g^F}}\theta\epsilon_H\right\}.
\end{equation}
\item The next iterate $x^{t+1}=x^t+\alpha_td^t$ satisfies 
\begin{equation}\label{suff-desc-sol}
F(x^t)-F(x^{t+1})\ge c_{\rm sol}\min\{\|\nabla F(x^{t+1})\|^2\epsilon_H^{-1},\epsilon_H^3\}.
\end{equation}
\end{enumerate}
\end{lemma}

\begin{proof}
One can observe that $F$ is descent along the iterates (whenever well-defined) generated by Algorithm~\ref{alg:NCG}, which together with $x^0=u^0$ implies that $F(x^t)\le F(u^0)$ and hence
$\|\nabla F(x^t)\|\le U^F_g$ due to \eqref{lwbd-Hgupbd}.
In addition, since $d^t$ results from the output $d$ of Algorithm~\ref{alg:capped-CG} with d$\_$type=SOL, one can see that $\|\nabla F(x^t)\| > \epsilon_g$ and \eqref{SOL-ppty-1}-\eqref{SOL-ppty-4} hold for $d^t$. Moreover, by $\|\nabla F(x^t)\| > \epsilon_g$ and \eqref{SOL-ppty-4}, one can conclude that $d^t\neq 0$.

We first prove statement (i). If \eqref{ls-sol} holds for $j=0$, then $\alpha_t=1$, which clearly implies that \eqref{lwbd-ak-sol} holds. We now suppose that \eqref{ls-sol} fails for $j=0$. Claim that for all $j\ge0$ that violate \eqref{ls-sol}, it holds that
\begin{equation}\label{lwbd-sol-ak}
\theta^{2j}\ge\min\{6(1-\eta),2\}\epsilon_H(L^F_H)^{-1}\|d^t\|^{-1}.
\end{equation}
Indeed, suppose that \eqref{ls-sol} is violated by some $j \ge 0$. We now show that \eqref{lwbd-sol-ak} holds for such $j$ by considering two separate cases below.

Case 1) $F(x^t+\theta^jd^t)>F(x^t)$. Let $\phi(\alpha)=F(x^t+\alpha d^t)$. Then $\phi(\theta^j)>\phi(0)$. Also, since $d^t\neq0$, by \eqref{SOL-ppty-1} and \eqref{SOL-ppty-3}, one has 
\[
\phi^\prime(0)  =\nabla F(x^t)^Td^t = -(d^t)^T(\nabla^2 F(x^t)+2\epsilon_HI)d^t 
 \le-\epsilon_H\|d^t\|^2< 0.
\]
Using these, we can observe that there exists a local minimizer $\alpha^*\in(0,\theta^j)$ of $\phi$ such that $\phi^\prime(\alpha^*)=\nabla F(x^t+\alpha^*d^t)^Td^t=0$ and $\phi(\alpha^*)<\phi(0)$, which implies that $F(x^t+\alpha^*d^t)< F(x^t)\le F(u^0)$. Hence, \eqref{apx-nxt-grad} holds for $x=x^t$ and $y=x^t+\alpha^*d^t$. Using this, $0<\alpha^*<\theta^j\le1$ and $\nabla F(x^t+\alpha^*d^t)^Td^t=0$, we obtain
\[
\begin{array}{l}
\frac{(\alpha^*)^2L_H^F}{2}\|d^t\|^3\overset{\eqref{apx-nxt-grad}}{\ge}\|d^t\|\|\nabla F(x^t+\alpha^*d^t)-\nabla F(x^t)-\alpha^*\nabla^2 F(x^t)d^t\|\\[6pt]
\ge(d^t)^T(\nabla F(x^t+\alpha^*d^t)-\nabla F(x^t)-\alpha^*\nabla^2 F(x^t)d^t)=-(d^t)^T\nabla F(x^t)-\alpha^*(d^t)^T\nabla^2 F(x^t)d^t\\[6pt]
\overset{\eqref{SOL-ppty-3}}{=}(1-\alpha^*)(d^t)^T(\nabla^2 F(x^t)+2\epsilon_HI)d^t +2 \alpha^*\epsilon_H\|d^t\|^2\overset{\eqref{SOL-ppty-1}}{\ge} (1+\alpha^*)\epsilon_H\|d^t\|^2\ge \epsilon_H \|d^t\|^2,
\end{array}
\]
which along with $d^t\neq 0$ implies that $(\alpha^*)^2\ge 2\epsilon_H(L_H^F)^{-1}\|d^t\|^{-1}$. Using this and $\theta^{j}>\alpha^*$, we conclude that \eqref{lwbd-sol-ak} holds in this case.

Case 2) $F(x^t+\theta^jd^t)\le F(x^t)$. This together with $F(x^t)\le F(u^0)$ implies that
\eqref{desc-ineq} holds for $x=x^t$ and $y=x^t+\theta^jd^t$. Then, because $j$ violates \eqref{ls-sol}, we obtain
\beqa
&&-\eta\epsilon_H\theta^{2j}\|d^t\|^2\le F(x^t+\theta^jd^t)-F(x^t)\overset{\eqref{desc-ineq}}{\le} \theta^j \nabla F(x^t)^Td^t + \frac{\theta^{2j}}{2}(d^t)^T\nabla^2 F(x^t) d^t + \frac{L_H^F}{6}\theta^{3j}\|d^t\|^3 \nn \\
&&\overset{\eqref{SOL-ppty-3}}{=}-\theta^j (d^t)^T(\nabla^2 F(x^t) +2\epsilon_H I)d^t + \frac{\theta^{2j}}{2}(d^t)^T\nabla^2 F(x^t) d^t + \frac{L_H^F}{6}\theta^{3j}\|d^t\|^3 \nn \\
&&=-\theta^j\left(1-\frac{\theta^j}{2}\right)(d^t)^T(\nabla^2 F(x^t) +2\epsilon_H I)d^t - \theta^{2j} \epsilon_H\|d^t\|^2 + \frac{L_H^F}{6}\theta^{3j}\|d^t\|^3 \nn \\
&&\overset{\eqref{SOL-ppty-1}}{\le} -\theta^j\left(1-\frac{\theta^j}{2}\right) \epsilon_H\|d^t\|^2 - \theta^{2j}\epsilon_H\|d^t\|^2 + \frac{L_H^F}{6}\theta^{3j}\|d^t\|^3 \le -\theta^j \epsilon_H\|d^t\|^2 + \frac{L_H^F}{6}\theta^{3j}\|d^t\|^3. \label{ineq:desc-ls-sol}
\eeqa
Recall that $d^t \neq 0$. Dividing both sides of \eqref{ineq:desc-ls-sol} by $L_H^F\theta^j\|d^t\|^3/6$ and using $\eta, \theta\in(0,1)$, we obtain that 
\[
\theta^{2j}\ge 6(1-\theta^j\eta)\epsilon_H(L_H^F)^{-1}\|d^t\|^{-1}\ge6(1-\eta)\epsilon_H(L_H^F)^{-1}\|d^t\|^{-1}.
\] 
 Hence, \eqref{lwbd-sol-ak} also holds in this case.

Combining the above two cases, we conclude that  \eqref{lwbd-sol-ak} holds for any $j\ge0$ that violates \eqref{ls-sol}.
By this and $\theta\in(0,1)$, one can see that all $j\ge0$ that violate \eqref{ls-sol} must be bounded above. It then follows that the step length $\alpha_t$ associated with \eqref{ls-sol} is well-defined. We next prove \eqref{lwbd-ak-sol}. Observe from the definition of $j_t$ in Algorithm~\ref{alg:NCG} that $j=j_t-1$ violates \eqref{ls-sol} and hence \eqref{lwbd-sol-ak} holds for $j=j_t-1$. Then, by \eqref{lwbd-sol-ak} with $j=j_t-1$ and $\alpha_t = \theta^{j_t}$, one has
\begin{equation}\label{lwbd-sol-ak-1}
\alpha_t = \theta^{j_t} \ge\sqrt{\min\{6(1-\eta),2\}\epsilon_H(L^F_H)^{-1}} \ \theta\|d^t\|^{-1/2},
\end{equation}
which,  along with \eqref{SOL-ppty-2} and $\|\nabla F(x^t)\|\le U^F_g$, implies \eqref{lwbd-ak-sol}. This proves statement (i).

We next prove statement (ii) 
by considering two separate cases below.

Case 1) $\alpha_t=1$. By this, one knows that \eqref{ls-sol} holds for $j=0$. It then follows that $F(x^t+d^t)\le F(x^t)\le F(u^0)$, which implies that \eqref{apx-nxt-grad} holds for $x=x^t$ and $y=x^t+d^t$. By this and \eqref{SOL-ppty-4}, one has
\[
\begin{array}{lcl}
\|\nabla F(x^{t+1})\|=\|\nabla F(x^{t}+d^t)\|
&\le&\|\nabla F(x^t+d^t)-\nabla F(x^t)-\nabla^2 F(x^t)d^t\| \\[4pt]
&&+\|(\nabla^2 F(x^t)+2\epsilon_HI)d^t+\nabla F(x^t)\|+2\epsilon_H\|d^t\|\\[4pt]
&\le& \frac{L_H^F}{2}\|d^t\|^2 + \frac{4+\zeta}{2}\epsilon_H\|d^t\|,
\end{array}
\]
where the last inequality follows from \eqref{apx-nxt-grad} and \eqref{SOL-ppty-4}. Solving the above inequality for $\|d^t\|$ and using the fact that $\|d^t\|>0$, we obtain that
\[
\begin{array}{rcll}
\|d^t\|&\ge&\frac{-(4+\zeta)\epsilon_H + \sqrt{(4+\zeta)^2\epsilon_H^2+8L_H^F\|\nabla F(x^{t+1})\|}}{2L_H^F}\ \ge\
\frac{-(4+\zeta)\epsilon_H + \sqrt{(4+\zeta)^2\epsilon_H^2+8L_H^F\epsilon_H^2}}{2L_H^F}\min\{\|\nabla F(x^{t+1})\|/\epsilon_H^2,1\}\\[6pt]
&=&\frac{4}{4+\zeta+\sqrt{(4+\zeta)^2+8L_H^F}}\min\{\|\nabla F(x^{t+1})\|/\epsilon_H,\epsilon_H\},
\end{array}
\]
where the second inequality follows from the inequality $-a+\sqrt{a^2+bs}\ge(-a+\sqrt{a^2+b})\min\{s,1\}$ for all $a,b,s\ge0$,
which can be verified by performing a rationalization to the terms $-a+\sqrt{a^2+b}$ and $-a+\sqrt{a^2+bs}$, respectively.
By this, $\alpha_t=1$, \eqref{ls-sol} and \eqref{csol}, one can see that \eqref{suff-desc-sol} holds.

Case 2) $\alpha_t<1$.
It then follows that $j=0$ violates \eqref{ls-sol} and hence \eqref{lwbd-sol-ak} holds for $j=0$.
Now, letting $j=0$ in \eqref{lwbd-sol-ak}, we obtain that $\|d^t\|\ge \min\{6(1-\eta),2\}\epsilon_H/L_H^F$,
 which together with \eqref{ls-sol} and \eqref{lwbd-sol-ak-1} implies that
\[
F(x^t)-F(x^{t+1}) \ge \eta \epsilon_H\theta^{2j_t}\|d^t\|^2 \ge \eta\frac{\min\{6(1-\eta),2\}\epsilon_H^2}{L_H^F}\theta^2\|d^t\|\ge \eta\left[\frac{\min\{6(1-\eta),2\}\theta}{L_H^F}\right]^2\epsilon_H^3.
\]
By this and \eqref{csol}, one can see that \eqref{suff-desc-sol} also holds in this case.
\end{proof}

The following lemma shows that when the search direction $d^t$ in Algorithm \ref{alg:NCG} is of type `NC', the line search step results in a sufficient reduction on $F$ as well.

\begin{lemma}\label{lem:nc}
Suppose that Assumption \ref{asp:NCG-cmplxity} holds and the direction $d^t$ results from  either the output $d$ of Algorithm~\ref{alg:capped-CG} with d$\_$type=NC or the output $v$ of Algorithm~\ref{pro:meo} at some iteration $t$ of Algorithm~\ref{alg:NCG}.
Let  $c_{\rm nc}$ be defined in \eqref{cnc}.
Then the following statements hold.
\begin{enumerate}[{\rm (i)}]
\item The step length $\alpha_t$ is well-defined, and $\alpha_t\ge \min\{1,\theta/L_H^F,3(1-\eta)\theta/L_H^F\}$.
\item The next iterate $x^{t+1}=x^t+\alpha_td^t$ satisfies $F(x^t)-F(x^{t+1})\ge c_{\rm nc}\epsilon_H^3$.
\end{enumerate}
\end{lemma}

\begin{proof}
Observe that $F$ is descent along the iterates (whenever well-defined) generated by Algorithm~\ref{alg:NCG}. Using this and $x^0=u^0$, we have $F(x^t)\le F(u^0)$.
By the assumption on $d^t$, one can see from Algorithm \ref{alg:NCG} that
$d^t$ is a negative curvature direction given in \eqref{dk-nc} or \eqref{dk-2nd-nc}. Also, notice that the vector $v$ returned from Algorithm~\ref{pro:meo} satisfies $\|v\|=1$. By these, Lemma~\ref{lem:SOL-NC-ppty}(ii), \eqref{dk-nc} and \eqref{dk-2nd-nc}, one can observe that
\begin{equation}\label{dk-Hk-dk3}
\nabla F(x^t)^Td^t\le 0, \quad (d^t)^T\nabla^2 F(x^t)d^t=-\|d^t\|^3<0.
\end{equation}

We first prove statement (i). If \eqref{ls-nc} holds for $j=0$, then $\alpha_t=1$, which clearly implies that $\alpha_t\ge\min\{1,\theta/L_H^F,3(1-\eta)\theta/L_H^F\}$. We now suppose that \eqref{ls-nc} fails for $j=0$. Claim that for all $j\ge0$ that violate \eqref{ls-nc}, it holds that
\begin{equation}\label{lwbd-at-nc}
\theta^j\ge \min\{1/L_H^F,3(1-\eta)/L_H^F\}.
\end{equation}
Indeed, suppose that \eqref{ls-nc} is violated by some $j \ge 0$. We now show that \eqref{lwbd-at-nc} holds for such $j$ by considering two separate cases below.

Case 1) $F(x^t+\theta^jd^t)>F(x^t)$. Let $\phi(\alpha)=F(x^t+\alpha d^t)$. Then $\phi(\theta^j)>\phi(0)$. Also, by  \eqref{dk-Hk-dk3}, one has 
\[
\phi^\prime(0)=\nabla F(x^t)^Td^t\le 0, \quad \phi^{\prime\prime}(0)=(d^t)^T\nabla^2 F(x^t)d^t<0.
\]
 Using these, we can observe that there exists a local minimizer $\alpha^*\in(0,\theta^j)$ of $\phi$ such that $\phi(\alpha^*)<\phi(0)$, namely,  $F(x^t+\alpha^*d^t)< F(x^t)$. By the second-order optimality condition of $\phi$ at $\alpha^*$, one has $\phi^{\prime\prime}(\alpha^*)=(d^t)^T\nabla^2 F(x^t+\alpha^*d^t)d^t\ge 0$. Since $F(x^t+\alpha^*d^t)<F(x^t)\le F(u^0)$, it follows that \eqref{F-Hess-Lip} holds for $x=x^t$ and $y=x^t+\alpha^*d^t$. Using this, the second relation in \eqref{dk-Hk-dk3} and $(d^t)^T\nabla^2 F(x^t+\alpha^*d^t)d^t\ge0$, we obtain that
in \eqref{dk-Hk-dk3} and $(d^t)^T\nabla^2 F(x^t+\alpha^*d^t)d^t\ge0$, we obtain that
\beqa
L_H^F\alpha^*\|d^t\|^3 &\overset{\eqref{F-Hess-Lip}}{\ge}&\|d^t\|^2\|\nabla^2 F(x^t+\alpha^*d^t)-\nabla^2 F(x^t)\| \ge (d^t)^T(\nabla^2 F(x^t+\alpha^*d^t)-\nabla^2 F(x^t))d^t \nn \\
&\ge& -(d^t)^T\nabla^2 F(x^t)d^t=\|d^t\|^3. \label{lwbd-dt-nc-j1}
\eeqa
Recall from \eqref{dk-Hk-dk3} that $d^t\neq 0$. It then follows from \eqref{lwbd-dt-nc-j1} that $\alpha^*\ge 1/L_H^F$, which along with $\theta^j>\alpha^*$ implies that $ \theta^j> 1/L_H^F$. Hence, \eqref{lwbd-at-nc} holds in this case.

Case 2) $F(x^t+\theta^jd^t)\le F(x^t)$. It follows from this and $F(x^t)\le F(u^0)$ that \eqref{desc-ineq} holds for $x=x^t$ and $y=x^t+\theta^jd^t$. By this and the fact that $j$ violates \eqref{ls-nc}, one has
\begin{equation*}\label{lwbd-dt-nc}
\begin{array}{lcl}
-\frac{\eta}{2}\theta^{2j}\|d^t\|^3 &\le& F(x^t+\theta^j d^t) - F(x^t) \overset{\eqref{desc-ineq}}{\le} \theta^j \nabla F(x^t)^Td^t + \frac{\theta^{2j}}{2}(d^t)^T\nabla^2 F(x^t) d^t + \frac{L_H^F}{6}\theta^{3j}\|d^t\|^3\\[6pt]
&\overset{\eqref{dk-Hk-dk3}}{\le}& -\frac{\theta^{2j}}{2}\|d^t\|^3 + \frac{L_H^F}{6}\theta^{3j}\|d^t\|^3,
\end{array}
\end{equation*}
which together with $d^t\neq 0$ implies that $\theta^j\ge 3(1-\eta)/L_H^F$. Hence, \eqref{lwbd-at-nc} also holds  in this case.

Combining the above two cases, we conclude that \eqref{lwbd-at-nc} holds for any $j\ge0$ that violates \eqref{ls-nc}. By this and $\theta\in(0,1)$, one can see that all $j\ge0$ that violate \eqref{ls-nc} must be bounded above. It then follows that the step length $\alpha_t$ associated with \eqref{ls-nc} is well-defined. We next derive a lower bound for $\alpha_t$. Notice from the definition of $j_t$ in Algorithm~\ref{alg:NCG} that $j=j_t-1$ violates \eqref{ls-nc} and hence \eqref{lwbd-at-nc} holds for $j=j_t-1$. Then, by \eqref{lwbd-at-nc} with $j=j_t-1$ and $\alpha_t = \theta^{j_t}$, one has $\alpha_t=\theta^{j_t} \ge \min\{\theta/L_H^F,3(1-\eta)\theta/L_H^F\}$, which immediately yields $\alpha_t\ge \min\{1,\theta/L_H^F,3(1-\eta)\theta/L_H^F\}$ as desired.

We next prove statement (ii) by considering two separate cases below.

Case 1) $d^t$ results from the output $d$ of Algorithm~\ref{alg:capped-CG} with d$\_$type=NC. It then follows from  \eqref{NC-ppty} that $\|d^t\|\ge\epsilon_H$. 
This together with  \eqref{ls-nc} and statement (i) implies that statement (ii) holds.

Case 2) $d^t$ results from the output $v$ of Algorithm~\ref{pro:meo}. Notice from
Algorithm~\ref{pro:meo} that $\|v\|=1$ and $v^T\nabla^2F(x^t)v\le-\epsilon_H/2$, which along with \eqref{dk-2nd-nc} yields $\|d^t\|\ge\epsilon_H/2$. By this, \eqref{ls-nc} and statement (i), one can see that statement (ii) again holds.
\end{proof}


\begin{proof}[Proof of Theorem \ref{thm:NCG-iter-oper-cmplxity}]
For notational convenience, we let $\{x^t\}_{t\in\bT}$ denote all the iterates generated by Algorithm~\ref{alg:NCG}, where $\bT$ is a set of consecutive nonnegative integers starting from $0$. Notice that $F$ is descent along the iterates generated by Algorithm~\ref{alg:NCG}, which together with $x^0=u^0$ implies that $x^t\in\{x:F(x)\le F(u^0)\}$. It then follows from
\eqref{lwbd-Hgupbd} that $\|\nabla^2 F(x^t)\|\le U^F_H$ holds for all $t\in\bT$.

(i) Suppose for contradiction that the total number of calls of Algorithm~\ref{pro:meo} in Algorithm~\ref{alg:NCG} is more than $T_2$. Notice from Algorithm~\ref{alg:NCG} and Lemma~\ref{lem:nc}(ii) that each of these calls, except the last one, returns a sufficiently negative curvature direction, and each of them results in a reduction on $F$ of at least $c_{\rm nc} \epsilon_H^3$. Hence, 
\[
T_2 c_{\rm nc} \epsilon_H^3 \le \sum_{t\in\bT} [F(x^t)-F(x^{t+1})] \le F(x^0)-\Fl = \Fh-\Fl,
\]
which contradicts the definition of $T_2$ given in \eqref{T1}. Hence, statement (i) of Theorem \ref{thm:NCG-iter-oper-cmplxity} holds.

(ii) Suppose for contradiction that the total number of calls of Algorithm~\ref{alg:capped-CG} in Algorithm~\ref{alg:NCG} is more than $T_1$. Observe that if Algorithm~\ref{alg:capped-CG} is called at some iteration $t$ and generates the next iterate $x^{t+1}$ satisfying $\|\nabla F(x^{t+1})\|\le\epsilon_g$, then Algorithm~\ref{pro:meo} must be called at the next iteration $t+1$. In view of this and statement (i) of Theorem \ref{thm:NCG-iter-oper-cmplxity}, we see that the total number of such iterations $t$ is at most $T_2$. Hence, the total number of iterations $t$ of Algorithm~\ref{alg:NCG} at which Algorithm~\ref{alg:capped-CG} is called and generates the next iterate $x^{t+1}$ satisfying $\|\nabla F(x^{t+1})\|>\epsilon_g$ is at least $T_1-T_2+1$. Moreover, for each of such iterations $t$, we observe from Lemmas~\ref{lem:sol}(ii) and \ref{lem:nc}(ii) that $F(x^t)-F(x^{t+1})\ge \min\{c_{\rm sol},c_{\rm nc}\}\min\{\epsilon_g^2\epsilon_H^{-1},\epsilon_H^3\}$. It then follows that 
\[
(T_1-T_2+1) \min\{c_{\rm sol},c_{\rm nc}\}\min\{\epsilon_g^2\epsilon_H^{-1},\epsilon_H^3\} \le \sum_{t\in\bT}[F(x^t)-F(x^{t+1})]\le \Fh-\Fl,
\]
which contradicts the definition of $T_1$ and $T_2$ given in \eqref{T1}. Hence, statement (ii) of Theorem \ref{thm:NCG-iter-oper-cmplxity} holds.

(iii) Notice that either Algorithm~\ref{alg:capped-CG} or \ref{pro:meo} is called at each iteration of Algorithm~\ref{alg:NCG}. It follows from this and statements (i) and (ii) of Theorem \ref{thm:NCG-iter-oper-cmplxity} that the total number of iterations of Algorithm~\ref{alg:NCG} is at most $T_1+T_2$. In addition, the relation \eqref{NCG-iter} follows from \eqref{csol}, \eqref{cnc} and \eqref{T1}. One can also observe that the output $x^t$ of Algorithm~\ref{alg:NCG} satisfies $\|\nabla F(x^t)\|\le \epsilon_g$ deterministically and $\lambda_{\min}(\nabla^2F(x^t))\ge-\epsilon_H$ with probability at least $1-\delta$ for some $0 \le t \le T_1+T_2$, where the latter part is due to
Algorithm~\ref{pro:meo}. This completes the proof of statement (ii) of Theorem \ref{thm:NCG-iter-oper-cmplxity}.

(iv) By Theorem~\ref{lem:capped-CG} with $(H,\varepsilon)=(\nabla^2 F(x^t),\epsilon_H)$ and the fact that $\|\nabla^2 F(x^t)\|\le U^F_H$, one can observe that  the number of Hessian-vector products required by each call of Algorithm~\ref{alg:capped-CG} with input $U = 0$ is at most $\widetilde{\cO}(\min\{n,(U_H^F/\epsilon_H)^{1/2}\})$. In addition, by Theorem~\ref{rand-Lanczos} with $(H,\varepsilon)=(\nabla^2 F(x^t),\epsilon_H)$, $\|\nabla^2 F(x^t)\|\le U^F_H$, and the fact that each iteration of the Lanczos method requires only one matrix-vector product, one can observe that the number of Hessian-vector products required by each call of Algorithm~\ref{pro:meo} is also at most $\widetilde{\cO}(\min\{n,(U_H^F/\epsilon_H)^{1/2}\})$. Based on these observations and statement (iii) of Theorem \ref{thm:NCG-iter-oper-cmplxity}, we see that  statement (iv) of this theorem holds.
\end{proof}

\subsection{Proof of the main results in Section \ref{sec:AL-method}}\label{sec:pf-AL}


Recall from Assumption~\ref{asp:lowbd-knownfeas}(a) that $\|c(z_{\epsilon_1})\|\le\epsilon_1/2<1$. By virtue of this, \eqref{lbd} and the definition of $\tilde{c}$ in
\eqref{model:equa-cnstr-pert},
we obtain that
\begin{equation}\label{ineq:p-lowbd}
f(x)+\gamma\|\tilde{c}(x)\|^2 \ge f(x)+ \gamma\|c(x)\|^2/2 - \gamma \|c(z_{\epsilon_1})\|^2 \ge \fl - \gamma,\quad \forall x\in\bR^n.
\end{equation}
We now prove the following auxiliary lemma that will be used frequently later.

\begin{lemma}\label{tech-1}
Suppose that Assumption \ref{asp:lowbd-knownfeas} holds. Let $\gamma$, $\fh$ and $\fl$ be given in Assumption \ref{asp:lowbd-knownfeas}.  Assume that $\rho>2\gamma$, $\lambda\in\bR^m$, and $x\in\bR^n$ satisfy
\begin{equation}\label{AL-upperbound}
\widetilde{\cL}(x,\lambda;\rho)\le \fh,
\end{equation}
where $\widetilde{\cL}$ is defined in \eqref{tL-tc}.
Then the following statements hold.
\begin{enumerate}[{\rm (i)}]
\item $f(x)\le \fh+\|\lambda\|^2/(2\rho)$.
\item $\|\tilde{c}(x)\|\le\sqrt{2(\fh-\fl+\gamma)/(\rho-2\gamma)+\|\lambda\|^2/(\rho-2\gamma)^2}+\|\lambda\|/(\rho-2\gamma)$.
\item If $\rho\ge\|\lambda\|^2/(2\tdf)$ for some $\tdf>0$, then $f(x)\le \fh+\tdf$.
\item If
\begin{equation}\label{bound-rho-deltac}
\rho\ge 2(\fh-\fl+\gamma)\tdc^{-2}+ 2\|\lambda\|\tdc^{-1}+2\gamma
\end{equation}
for some $\tdc>0$, then $\|\tilde{c}(x)\|\le\tdc$.
\end{enumerate}
\end{lemma}
	
\begin{proof}
(i) It follows from
 \eqref{AL-upperbound} and the definition of $\widetilde{\cL}$ in \eqref{tL-tc}
that
\[
\textstyle
\fh\ge f(x)+\lambda^T \tilde{c}(x)+\frac{\rho}{2}\|\tilde{c}(x)\|^2 = f(x)+\frac{\rho}{2}\left\|\tilde{c}(x)+\frac{\lambda}{\rho}\right\|^2-\frac{\|\lambda\|^2}{2\rho} \ge f(x)-\frac{\|\lambda\|^2}{2\rho}.
\]
Hence, statement (i) holds.

(ii) In view of \eqref{ineq:p-lowbd} and \eqref{AL-upperbound}, one has
\begin{equation*}\label{ineq:bound-cx-step1}
\begin{array}{l}
\fh\overset{\eqref{AL-upperbound}}{\ge}f(x)+\lambda^T\tilde{c}(x)+\frac{\rho}{2}\|\tilde{c}(x)\|^2=f(x)+\gamma\|\tilde{c}(x)\|^2
+\frac{\rho-2\gamma}{2}\left\|\tilde{c}(x)+\frac{\lambda}{\rho-2\gamma}\right\|^2-\frac{\|\lambda\|^2}{2(\rho-2\gamma)}\\[2pt]
\ \ \ \ \overset{\eqref{ineq:p-lowbd}}{\ge}\fl- \gamma+\frac{\rho-2\gamma}{2}\left\|\tilde{c}(x)+\frac{\lambda}{\rho-2\gamma}\right\|^2-\frac{\|\lambda\|^2}{2(\rho-2\gamma)}.
\end{array}
\end{equation*}
It then follows that $\left\|\tilde{c}(x)+\frac{\lambda}{\rho-2\gamma}\right\|\le\sqrt{\frac{2(\fh-\fl+\gamma)}{\rho-2\gamma}+\frac{\|\lambda\|^2}{(\rho-2\gamma)^2}}$,  which implies that statement (ii) holds.

(iii) Statement (iii) immediately follows from statement (i) and $\rho\ge\|\lambda\|^2/(2\tdf)$.

(iv) Suppose that \eqref{bound-rho-deltac} holds. Multiplying both sides of \eqref{bound-rho-deltac} by $\tdc^{2}$ and rearranging the terms, we have
\[
(\rho-2\gamma)\tdc^2-2\|\lambda\|\tdc-2(\fh-\fl+\gamma)\ge0.
\]
Recall that $\rho>2\gamma$ and $\tdc>0$. Solving this inequality for $\tdc$ yields 
\[
\tdc\ge\sqrt{2(\fh-\fl+\gamma)/(\rho-2\gamma)+\|\lambda\|^2/(\rho-2\gamma)^2}+\|\lambda\|/(\rho-2\gamma),
\]
which along with statement (ii) implies that $\|\tilde{c}(x)\|\le\tdc$. Hence, statement (iv) holds.
\end{proof}


\begin{proof}[Proof of Lemma \ref{lem:level-set-augmented-lagrangian-func}]
(i) Let $x$ be any point such that $\widetilde{\cL}(x,\lambda^k;\rho_k)\le \widetilde{\cL}(x^{k}_{\init},\lambda^k;\rho_k)$. It then follows from \eqref{L-xinit} that $\widetilde{\cL}(x,\lambda^k;\rho_k) \le \fh$. By this, $\|\lambda^k\|\le\Lambda$, $\rho_k \ge \rho_0>2\gamma$, $\delta_{f,1} \le \bdelta_f$, $\delta_{c,1} \le \bdelta_c$, and Lemma \ref{tech-1} with $(\lambda,\rho)=(\lambda^k,\rho_k)$,  one has 
\begin{align}
f(x) & \le \fh+ \|\lambda^k\|^2/(2\rho_k)\le \fh+ \Lambda^2/(2\rho_0)=\fh+\delta_{f,1} \le \fh+\bdelta_f, \nonumber \\
\textstyle\|\tilde{c}(x)\|&\textstyle\le\sqrt{\frac{2(\fh-\fl+\gamma)}{\rho_k-2\gamma}+\frac{\|\lambda^k\|^2}{(\rho_k-2\gamma)^2}}+\frac{\|\lambda^k\|}{\rho_k-2\gamma}\le\sqrt{\frac{2(\fh-\fl+\gamma)}{\rho_0-2\gamma}+\frac{\Lambda^2}{(\rho_0-2\gamma)^2}}+\frac{\Lambda}{\rho_0-2\gamma}=\delta_{c,1} \le \bdelta_c. \label{ineq:feasible-bd}
\end{align}
Also, recall from the definition of $\tilde{c}$ in \eqref{model:equa-cnstr-pert}
and $\|c(z_{\epsilon_1})\|\le1$ that $\|c(x)\|\le1+\|\tilde{c}(x)\|$. This together with the above inequalities and \eqref{nearly-feas-level-set} implies $x\in\cS(\delta_{f},\delta_{c})$. Hence, statement (i) of Lemma \ref{lem:level-set-augmented-lagrangian-func} holds.

(ii) Note that $\inf\limits_{x\in\bR^n} \widetilde{\cL}(x,\lambda^k;\rho_k)=\inf\limits_{x\in\bR^n}\{\widetilde{\cL}(x,\lambda^k;\rho_k):\widetilde{\cL}(x,\lambda^k;\rho_k) \le \widetilde{\cL}(x^{k}_{\init},\lambda^k;\rho_k)\}$. Consequently, to prove statement (ii) of Lemma \ref{lem:level-set-augmented-lagrangian-func}, it suffices to show that
\begin{equation} \label{L-lwbd}
\inf\limits_{x\in\bR^n}\{\widetilde{\cL}(x,\lambda^k;\rho_k): \widetilde{\cL}(x,\lambda^k;\rho_k) \le \widetilde{\cL}(x^{k}_{\init},\lambda^k;\rho_k)\} \ge\fl - \gamma-\Lambda\bdelta_c.
\end{equation}
To this end, let $x$ be any point satisfying $\widetilde{\cL}(x,\lambda^k;\rho_k) \le\widetilde{\cL}(x^{k}_{\init},\lambda^k;\rho_k)$. We then know from \eqref{ineq:feasible-bd} that $\|\tilde{c}(x)\|\le\bdelta_c$. By this, $\|\lambda^k\|\le\Lambda$, $\rho_k>2\gamma$, and \eqref{ineq:p-lowbd}, one has
\begin{equation*}\label{ineq:lowerbound-ALfunc} 
\begin{array}{l}
\widetilde{\cL}(x,\lambda^k;\rho_k) 
=f(x)+\gamma\|\tilde{c}(x)\|^2+(\lambda^k)^T\tilde{c}(x)+\frac{\rho_k-2\gamma}{2}\|\tilde{c}(x)\|^2\\[4pt]
\ge f(x)+\gamma\|\tilde{c}(x)\|^2 - \Lambda\|\tilde{c}(x)\| \ge\fl - \gamma-\Lambda\bdelta_c,
\end{array}
\end{equation*}
and hence \eqref{L-lwbd} holds as desired.
\end{proof}


\begin{proof}[Proof of Theorem \ref{thm:output-alg1}]
Suppose that Algorithm \ref{alg:2nd-order-AL-nonconvex} terminates at some iteration $k$, that is, $\tau_k^g\le\epsilon_1$, $\tau_k^H\le\epsilon_2$, and $\|c(x^{k+1})\|\le\epsilon_1$ hold. Then, by $\tau_k^g\le\epsilon_1$, $\tilde{\lambda}^{k+1}=\lambda^k+\rho_k\tilde{c}(x^{k+1})$, $\nabla \tilde{c}=\nabla c$ and the second relation in \eqref{algstop:1st-order}, one has 
 \begin{align*}
\|\nabla f(x^{k+1})+\nabla c(x^{k+1})\tilde{\lambda}^{k+1}\| &=\|\nabla f(x^{k+1})+\nabla \tilde{c}(x^{k+1})(\lambda^k+\rho_k \tilde{c}(x^{k+1}))\| \\ 
&=\|\nabla_x\widetilde{\cL}(x^{k+1},\lambda^k;\rho_k)\|\le \tau_k^g \le \epsilon_1.
\end{align*} 
 Hence, $(x^{k+1},\tilde{\lambda}^{k+1})$ satisfies the first relation in \eqref{optcond:1st-equa-cnstr}. In addition, by \eqref{algstop:2nd-order} and $\tau_k^H\le\epsilon_2$, one can show that $\lambda_{\min}(\nabla^2_{xx}\widetilde{\cL}(x^{k+1},\lambda^k;\rho_k))\ge-\epsilon_2$ with probability at least $1-\delta$, which leads to $d^T\nabla^2_{xx}\widetilde{\cL}(x^{k+1},\lambda^k;\rho_k)d \ge -\epsilon_2\|d\|^2$ for all $d\in\bR^n$ with probability at least $1-\delta$.  Using this, $\tilde{\lambda}^{k+1}=\lambda^k+\rho_k\tilde{c}(x^{k+1})$, $\nabla \tilde{c} = \nabla c$, and $\nabla^2 \tilde{c}_i=\nabla^2 c_i$ for $1\le i\le m$, we see that with probability at least $1-\delta$, it holds that 
\[
d^T\left(\nabla^2 f(x^{k+1})+\sum_{i=1}^m\tilde{\lambda}_i^{k+1}\nabla^2c_i(x^{k+1})+\rho_k\nabla c(x^{k+1})\nabla c(x^{k+1})^T\right)d\ge -\epsilon_2\|d\|^2 \ \  \forall d\in\bR^n,
\]
which implies $d^T(\nabla^2 f(x^{k+1})+\sum_{i=1}^m\tilde{\lambda}_i^{k+1}\nabla^2 c_i(x^{k+1}))d\ge -\epsilon_2\|d\|^2$ for all $d\in \cC(x^{k+1})$, where $\cC(\cdot)$ is defined in \eqref{def:critical-cone}. Hence, $(x^{k+1},\tilde{\lambda}^{k+1})$ satisfies \eqref{optcond:2nd-equa-cnstr} with probability at least $1-\delta$. Combining these with $\|c(x^{k+1})\|\le\epsilon_1$, we conclude that $x^{k+1}$ is a deterministic $\epsilon_1$-FOSP of \eqref{model:equa-cnstr} and an $(\epsilon_1,\epsilon_2)$-SOSP of \eqref{model:equa-cnstr} with probability at least $1-\delta$. Hence, Theorem \ref{thm:output-alg1} holds.
\end{proof}


\begin{proof}[Proof of Theorem \ref{thm:out-itr-cmplxity-1}]
It follows from \eqref{def:delta0c-rhobar1} that ${\rho}_{\epsilon_1}\ge 2\rho_0$.
By this, one has
\begin{equation}\label{T-epsilong}
K_{\epsilon_1} \ \overset{\eqref{T-epsilon-g}}{=} \ \left\lceil\log\epsilon_1/\log \omega_1 \right\rceil \ \overset{\eqref{omega-tolerance}}{=} \ \left\lceil\log{2}/\log{r}\right\rceil  \
\le \ \log({\rho}_{\epsilon_1}\rho_0^{-1})/\log{r}+1.
\end{equation}
Notice that $\{\rho_k\}$ is either unchanged or increased by a ratio $r$ as $k$ increases.
By this fact and \eqref{T-epsilong}, we see that
\begin{equation}\label{rho:upper-bound-Tepsilong}
\max_{0 \le k \le K_{\epsilon_1}} \rho_k \le{r}^{K_{\epsilon_1}}\rho_0\overset{\eqref{T-epsilong}}{\le}{r}^{\frac{\log({\rho}_{\epsilon_1}\rho_0^{-1})}{\log{r}}+1}\rho_0={r}{\rho}_{\epsilon_1}.
\end{equation}
In addition, notice that $\rho_k>2\gamma$ and $\|\lambda^k\|\le \Lambda$. Using these, \eqref{hbd}, the first relation in \eqref{algstop:1st-order}, and Lemma \ref{tech-1}(ii) with $(x,\lambda,\rho)=(x^{k+1},\lambda^k,\rho_k)$,  we obtain that
\begin{equation}\label{bd-ck}\textstyle
\|\tilde{c}(x^{k+1})\| \le\sqrt{\frac{2(\fh-\fl+\gamma)}{\rho_k-2\gamma}+\frac{\|\lambda^k\|^2}{(\rho_k-2\gamma)^2}}+\frac{\|\lambda^k\|}{\rho_k-2\gamma}\le  \sqrt{\frac{2(\fh-\fl+\gamma)}{\rho_k-2\gamma}+\frac{\Lambda^2}{(\rho_k-2\gamma)^2}}+\frac{\Lambda}{\rho_k-2\gamma}.
\end{equation}
Also, we observe from $\|c(z_{\epsilon_1})\|\le\epsilon_1/2$ and the definition of $\tilde{c}$ in
\eqref{model:equa-cnstr-pert} that
\begin{equation}\label{bd-cxk}
\|c(x^{k+1})\|\le\|\tilde{c}(x^{k+1})\|+\|c(z_{\epsilon_1})\|\le \|\tilde{c}(x^{k+1})\|+\epsilon_1/2.
\end{equation}

We now prove that $\overline{K}_{\epsilon_1}$ is finite. Suppose for contradiction that $\overline{K}_{\epsilon_1}$ is infinite. It then follows from this and \eqref{number-outer-iteration} that $\|c(x^{k+1})\|>\epsilon_1$ for all $k\ge K_{\epsilon_1}$, which along with \eqref{bd-cxk} implies that $\|\tilde{c}(x^{k+1})\|>\epsilon_1/2$ for all $k\ge K_{\epsilon_1}$. It then follows that $\|\tilde{c}(x^{k+1})\|>\alpha\|\tilde{c}(x^k)\|$ must hold for infinitely many $k$'s. Using this and the update scheme on $\{\rho_k\}$, we deduce that $\rho_{k+1}={r}\rho_k$ holds for infinitely many $k$'s, which together with the monotonicity of $\{\rho_k\}$ implies that
$\rho_k\to\infty$ as $k\to\infty$. By this and \eqref{bd-ck}, one can see that  $\|\tilde{c}(x^{k+1})\|\to0$ as $k\to\infty$, which contradicts the fact that $\|\tilde{c}(x^{k+1})\|>\epsilon_1/2$ holds for all $k\ge K_{\epsilon_1}$. Hence, $\overline{K}_{\epsilon_1}$ is finite.
In addition, notice from \eqref{omega-tolerance}, \eqref{T-epsilon-g} and \eqref{epsw1-epsw2} that $(\tau_k^g,\tau_k^H)=(\epsilon_1,\epsilon_2)$ for all $k\ge K_{\epsilon_1}$. This along with the termination criterion of Algorithm \ref{alg:2nd-order-AL-nonconvex}  and the definition of $\overline{K}_{\epsilon_1}$ implies that Algorithm \ref{alg:2nd-order-AL-nonconvex} must terminate at iteration $\overline{K}_{\epsilon_1}$.

We next show that \eqref{outer-iteration-cmplxty} and $\rho_k\le{r}{\rho}_{\epsilon_1}$ hold for $0 \le k \le \overline{K}_{\epsilon_1}$ by considering two separate cases below.

Case 1) $\|c(x^{K_{\epsilon_1}+1})\|\le\epsilon_1$.  By this and \eqref{number-outer-iteration}, one can see that  $\overline{K}_{\epsilon_1}=K_{\epsilon_1}$, which together with \eqref{T-epsilong} and \eqref{rho:upper-bound-Tepsilong} implies that \eqref{outer-iteration-cmplxty} and $\rho_k\le{r}{\rho}_{\epsilon_1}$ hold for $0 \le k \le \overline{K}_{\epsilon_1}$.

Case 2) $\|c(x^{K_{\epsilon_1}+1})\|>\epsilon_1$. By this and \eqref{number-outer-iteration}, one can observe that $\overline{K}_{\epsilon_1}>K_{\epsilon_1}$  and also $\|c(x^{k+1})\|>\epsilon_1$ for all $K_{\epsilon_1} \le k \le \overline{K}_{\epsilon_1}-1$, which together with \eqref{bd-cxk} implies
\begin{equation}\label{lbd-tc}
\|\tilde{c}(x^{k+1})\|>\epsilon_1/2,\quad \forall K_{\epsilon_1} \le k \le \overline{K}_{\epsilon_1}-1.
\end{equation}
It then follows from $\|\lambda^k\|\le\Lambda$, \eqref{hbd}, the first relation in \eqref{algstop:1st-order}, and Lemma \ref{tech-1}(iv) with $(x,\lambda,\rho,\tdc)=(x^{k+1},\lambda^k,\rho_k,\epsilon_1/2)$ that
\begin{equation}\label{bound-rhok}
\begin{aligned}
\rho_k &< 8(\fh-\fl+\gamma)\epsilon_1^{-2}+4\|\lambda^k\|\epsilon_1^{-1}+2\gamma\\
&\le 8(\fh-\fl+\gamma)\epsilon_1^{-2}+4\Lambda\epsilon_1^{-1}+2\gamma\overset{\eqref{def:delta0c-rhobar1}}{\le}{\rho}_{\epsilon_1},\quad
\forall K_{\epsilon_1} \le k \le \overline{K}_{\epsilon_1}-1.
\end{aligned}
\end{equation}
Combining this relation, \eqref{rho:upper-bound-Tepsilong}, and the fact $\rho_{\overline{K}_{\epsilon_1}}\le {r}\rho_{\overline{K}_{\epsilon_1}-1}$, we conclude that  $\rho_k\le{r}{\rho}_{\epsilon_1}$ holds for $0 \le k\le \overline{K}_{\epsilon_1}$. It remains to
show that \eqref{outer-iteration-cmplxty} holds.  To this end, let 
\[
\bK =\{k:\rho_{k+1}={r}\rho_{k}, K_{\epsilon_1}\le k\le \overline{K}_{\epsilon_1}-2\}.
\]
It follows from \eqref{bound-rhok} and the update scheme of $\rho_k$ that 
\[
{r}^{|\bK|}\rho_{K_{\epsilon_1}}=\max_{K_{\epsilon_1}\le k \le \overline{K}_{\epsilon_1}-1}\{\rho_k\}\le {\rho}_{\epsilon_1},
\]
which together with $\rho_{K_{\epsilon_1}}\ge\rho_0$ implies that
\begin{equation}\label{bound-cU}
|\bK|\le \log({\rho}_{\epsilon_1}\rho_{K_{\epsilon_1}}^{-1})/ \log{r} \le \log({\rho}_{\epsilon_1}\rho_0^{-1})/ \log{r}.
\end{equation}
Let $\{k_1,k_2,\ldots,k_{|\bK|}\}$ denote all the elements of $\mathcal{\bK}$ arranged in ascending order, and let $k_0=K_{\epsilon_1}$ and $k_{|\bK|+1}=\overline{K}_{\epsilon_1}-1$. We next derive an upper bound for $k_{j+1}-k_{j}$ for $j=0,1,\ldots,|\bK|$. By the definition of $\bK$, one can observe that $\rho_{k}=\rho_{k'}$ for $k_j< k,k'\le k_{j+1}$. Using this and the update scheme of $\rho_k$, we deduce that
\begin{equation}\label{ineq:sequence-cxk}
\|\tilde{c}(x^{k+1})\|\le\alpha\|\tilde{c}(x^{k})\|,\quad \forall k_j<k< k_{j+1}.
\end{equation}
On the other hand, by \eqref{def:delta0c-rhobar1xxx}, \eqref{bd-ck} and $\rho_k\ge\rho_0$, one has $\|\tilde{c}(x^{k+1})\|\le \delta_{c,1}$ for $0 \le k \le \overline{K}_{\epsilon_1}$. By this and \eqref{lbd-tc}, one can see that
\begin{equation}\label{ineq:sequence-cxk2222}
\epsilon_1/2<\|\tilde{c}(x^{k+1})\|\le\delta_{c,1},\quad \forall K_{\epsilon_1} \le k \le \overline{K}_{\epsilon_1}-1.
\end{equation}
Now, note that either $k_{j+1} - k_j = 1$ or $k_{j+1} - k_j > 1$. In the latter case, we can apply \eqref{ineq:sequence-cxk} with $k = k_{j+1}-1,\ldots,k_j+1$ together with \eqref{ineq:sequence-cxk2222} to deduce that
\[
\epsilon_1/2<\|\tilde{c}(x^{k_{j+1}})\|\le\alpha \|\tilde{c}(x^{k_{j+1}-1})\|\le\cdots\le \alpha^{k_{j+1}-k_j-1}\|\tilde{c}(x^{k_{j}+1})\|\le \alpha^{k_{j+1}-k_j-1}\delta_{c,1},\quad \forall j=0,1,\ldots,|\bK|.
\]
Combining these two cases, we have 
\begin{equation}\label{ineq:subsequence-rho-same}
k_{j+1}-k_{j}\le |\log(\epsilon_1(2\delta_{c,1})^{-1}))/\log \alpha|+1,\quad \forall j=0,1,\ldots,|\bK|.
\end{equation}
Summing up these inequalities, and using \eqref{T-epsilong}, \eqref{bound-cU},  $k_0=K_{\epsilon_1}$ and $k_{|\bK|+1}=\overline{K}_{\epsilon_1}-1$, we have
\begin{eqnarray}
&\hspace{-.85in}\overline{K}_{\epsilon_1} = 1+k_{|\bK|+1}=1+k_0+\sum_{j=0}^{|\bK|}(k_{j+1}-k_{j}) \overset{\eqref{ineq:subsequence-rho-same}}{\le}1+K_{\epsilon_1}+(|\bK|+1)\left(\left|\frac{\log(\epsilon_1(2\delta_{c,1})^{-1})}{\log \alpha}\right|+1\right) \nonumber \\[6pt]
&\le2+\frac{\log({\rho}_{\epsilon_1}\rho_0^{-1})}{\log{r}}+\left(\frac{\log({\rho}_{\epsilon_1}\rho_0^{-1})}{\log{r}}+1\right)\left(\left|\frac{\log(\epsilon_1(2\delta_{c,1})^{-1})}{\log \alpha}\right|+1\right) =1+ \left(\frac{\log({\rho}_{\epsilon_1}\rho_0^{-1})}{\log{r}}+1\right)\left(\left|\frac{\log(\epsilon_1(2\delta_{c,1})^{-1})}{\log \alpha}\right|+2\right), \nn
\end{eqnarray}
where the second inequality is due to \eqref{T-epsilong} and \eqref{bound-cU}. Hence, \eqref{outer-iteration-cmplxty} also holds in this case.
\end{proof}

We next prove Theorem \ref{thm:total-iter-cmplxity}. Before proceeding, we introduce some notation that will be used shortly. Let $L_{k,H}$ denote the Lipschitz constant of $\nabla^2_{xx} \widetilde{\cL}(x,\lambda^k;\rho_k)$ on the convex open neighborhood $\Omega(\bdelta_{f},\bdelta_{c})$ of
$\cS(\bdelta_f,\bdelta_c)$, where $\cS(\bdelta_f,\bdelta_c)$ is defined in \eqref{nearly-feas-level-set}, and let $U_{k,H}=\sup_{x\in\cS(\bdelta_f,\bdelta_c)} \|\nabla^2_{xx} \widetilde{\cL}(x,\lambda^k;\rho_k)\|$. Notice from \eqref{model:equa-cnstr-pert} and \eqref{tL-tc} that
\begin{equation}\label{eq:hessian-augmented-Lagrangian}
\nabla^2_{xx}\widetilde{\cL}(x,\lambda^k;\rho_k) = \nabla^2 f(x) +\sum_{i=1}^m\lambda^k_{i}\nabla^2 c_i(x) + \rho_k \bigg(\nabla c(x)\nabla c(x)^T + \sum_{i=1}^m \tilde{c}_i(x)\nabla^2 c_i(x)\bigg).
\end{equation}
By this, $\|\lambda^k\| \le \Lambda$, the definition of $\tilde{c}$, and the Lipschitz continuity of $\nabla^2 f$ and $\nabla^2 c_i$'s (see Assumption~\ref{asp:lowbd-knownfeas}(c)), one can observe that there exist some constants ${L}_{1}$, ${L}_{2}$, ${U}_{1}$ and ${U}_{2}$, depending only on $f$, $c$, $\Lambda$, $\bdelta_{f}$ and $\bdelta_{c}$,  such that
\begin{equation}\label{ineq:bound-norm-hessian-Liphessian}
L_{k,H}\le {L}_{1}+\rho_k{L}_{2},\quad U_{k,H}\le{U}_{1}+\rho_k{U}_{2}.
\end{equation}


\begin{proof}[Proof of Theorem \ref{thm:total-iter-cmplxity}]
Let $T_k$ and $N_k$ denote the number of iterations and matrix-vector products performed by Algorithm~\ref{alg:NCG} at the outer iteration $k$ of Algorithm \ref{alg:2nd-order-AL-nonconvex}, respectively. It then follows from Theorem \ref{thm:out-itr-cmplxity-1} that the total number of iterations and matrix-vector products performed by Algorithm~\ref{alg:NCG} in Algorithm \ref{alg:2nd-order-AL-nonconvex} are $\sum_{k=0}^{\overline{K}_{\epsilon_1}}T_k$ and $\sum_{k=0}^{\overline{K}_{\epsilon_1}}N_k$, respectively. In addition, notice from \eqref{def:delta0c-rhobar1} and Theorem \ref{thm:out-itr-cmplxity-1} that ${\rho}_{\epsilon_1}=\cO(\epsilon_1^{-2})$ and  $\rho_k\le{r}{\rho}_{\epsilon_1}$, which yield $\rho_k=\cO(\epsilon_1^{-2})$.

We first claim that $(\tau_k^g)^2/\tau_k^H\ge\min\{\epsilon_1^2/\epsilon_2,\epsilon_2^3\}$ holds for any $k\ge0$. Indeed, let $\bar t=\log\epsilon_1/\log\omega_1$ and $\psi(t)=\max\{\epsilon_1,\omega_1^t\}^2/\max\{\epsilon_2,\omega_2^t\}$ for all $t\in\bR$. It then follows from \eqref{epsw1-epsw2} that $\omega_1^{\bar t}=\epsilon_1$ and $\omega_2^{\bar t}=\epsilon_2$. By this and $\omega_1,\omega_2\in(0,1)$, one can observe that  $\psi(t)=(\omega_1^2/\omega_2)^t$ if $t \leq \bar t$ and $\psi(t)=\epsilon_1^2/\epsilon_2$ otherwise. This along with $\epsilon_2\in(0,1)$ implies that
\[
\min_{t\in[0,\infty)}\psi(t) = \min\{\psi(0),\psi(\bar t)\}=\min\{1,\epsilon_1^2/\epsilon_2\} \geq \min\{\epsilon_1^2/\epsilon_2,\epsilon_2^3\},
\]
which together with \eqref{omega-tolerance} yields
$(\tau_k^g)^2/\tau_k^H=\psi(k) \geq \min\{\epsilon_1^2/\epsilon_2,\epsilon_2^3\}$ for all $k \geq 0$.

(i) From Lemma \ref{lem:level-set-augmented-lagrangian-func}(i) and the definitions of $\Omega(\bdelta_{f},\bdelta_{c})$ and $L_{k,H}$, we see that $L_{k,H}$ is a Lipschitz constant of $\nabla^2_{xx} \widetilde{\cL}(x,\lambda^k;\rho_k)$ on a convex open neighborhood  of $\{x:\widetilde{\cL}(x,\lambda^k;\rho_k)\le \widetilde{\cL}(x^{k}_{\init},\lambda^k;\rho_k)\}$.
Also, recall from Lemma \ref{lem:level-set-augmented-lagrangian-func}(ii) that $\inf_{x\in\bR^n} \widetilde{\cL}(x,\lambda^k;\rho_k) $ $\ge\fl-\gamma-\Lambda\bdelta_c$. By these, $\widetilde{\cL}(x^{k}_{\init},\lambda^k;\rho_k)\le \fh$ (see \eqref{L-xinit}) and Theorem \ref{thm:NCG-iter-oper-cmplxity}(iii) with $(\Fh,\Fl,L_{H}^F,\epsilon_g,\epsilon_H)=(\widetilde{\cL}(x^{k}_{\init},\lambda^k;\rho_k),\fl-\gamma-\Lambda \bdelta_c,L_{k,H},\tau_k^g,\tau_k^H)$, one has
\begin{equation}\label{eq:alg1-inner-iter-kthAL}
\begin{array}{rcl}
T_k&=&\cO((\fh-\fl+\gamma+\Lambda\bdelta_{c})L_{k,H}^2\max\{(\tau^g_k)^{-2}\tau^H_k,(\tau_k^H)^{-3}\})\\[2pt]
&\overset{\eqref{ineq:bound-norm-hessian-Liphessian}}{=}&\cO(\rho_k^2\max\{(\tau^g_k)^{-2}\tau^H_k,(\tau_k^H)^{-3}\})=\cO(\epsilon_1^{-4}\max\{\epsilon_1^{-2}\epsilon_2,\epsilon_2^{-3}\}),
\end{array}
\end{equation}
where the last equality is from $(\tau_k^g)^2/\tau_k^H\ge\min\{\epsilon_1^2/\epsilon_2,\epsilon_2^3\}$, $\tau_k^H\ge\epsilon_2$, and $\rho_k=\cO(\epsilon_1^{-2})$.

Next, if 
$c(x)=Ax-b$ for some $A\in\bR^{m\times n}$ and $b\in\bR^m$, then $\nabla c(x)=A^T$ and $\nabla^2 c_i(x)=0$ for $1 \le i \le m$. By these and \eqref{eq:hessian-augmented-Lagrangian}, one has $L_{k,H}=\cO(1)$. Using this and similar arguments as for \eqref{eq:alg1-inner-iter-kthAL}, we obtain that $T_k=\cO(\max\{\epsilon_1^{-2}\epsilon_2,\epsilon_2^{-3}\})$. By this, \eqref{eq:alg1-inner-iter-kthAL} and $\overline{K}_{\epsilon_1}=\cO(|\log\epsilon_1|^2)$ (see Remark \ref{order-out-itera-penalty}), we conclude that statement (i) of Theorem \ref{thm:total-iter-cmplxity} holds.

 (ii) In view of Lemma \ref{lem:level-set-augmented-lagrangian-func}(i) and the definition of
 $U_{k,H}$, one can see that 
\[
U_{k,H} \ge \sup_{x\in\bR^n}\{\|\nabla^2_{xx} \widetilde{\cL}(x,\lambda^k;\rho_k)\|: \widetilde{\cL}(x,\lambda^k;\rho_k)\le \widetilde{\cL}(x^{k}_{\init},\lambda^k;\rho_k)\}.
\] 
 Using this, $\widetilde{\cL}(x^{k}_{\init},\lambda^k;\rho_k)$ $\le \fh$ and Theorem \ref{thm:NCG-iter-oper-cmplxity}(iv) with $(\Fh,\Fl,L_{H}^F,U_{H}^F,\epsilon_g,\epsilon_H)=(\widetilde{\cL}(x^{k}_{\init},\lambda^k;\rho_k),
\fl-\gamma-\Lambda\bdelta_c,L_{k,H},U_{k,H},\tau_k^g,\tau_k^H)$, we obtain that
\begin{equation}\label{eq:alg1-oper-kthAL}
\begin{array}{rcl}
N_k&=&\widetilde{\cO}((\fh-\fl+\gamma+\Lambda\bdelta_c)L_{k,H}^2\max\{(\tau^g_k)^{-2}\tau^H_k,(\tau_k^H)^{-3}\}\min\{n,(U_{k,H}/\tau_k^H)^{1/2}\})\\[4pt]
&\overset{\eqref{ineq:bound-norm-hessian-Liphessian}}{=}&\widetilde{\cO}(\rho_k^2\max\{(\tau^g_k)^{-2}\tau^H_k,(\tau_k^H)^{-3}\}\min\{n,(\rho_k/\tau_k^H)^{1/2}\})\\[4pt]
&=&\widetilde{\cO}(\epsilon_1^{-4}\max\{\epsilon_1^{-2}\epsilon_2,\epsilon_2^{-3}\}\min\{n,\epsilon_1^{-1}\epsilon_2^{-1/2}\}),	
\end{array}
\end{equation}
where the last equality  is from $(\tau_k^g)^2/\tau_k^H\ge\min\{\epsilon_1^2/\epsilon_2,\epsilon_2^3\}$, $\tau_k^H\ge\epsilon_2$, and $\rho_k=\cO(\epsilon_1^{-2})$.

On the other hand, if $c$ is assumed to be affine, it follows from the above discussion that $L_{k,H}=\cO(1)$. Using this, $U_{k,H}\le U_1+\rho_k U_2$, and similar arguments as for \eqref{eq:alg1-oper-kthAL}, we obtain that $N_k=\widetilde{\cO}(\max\{\epsilon_1^{-2}\epsilon_2,\epsilon_2^{-3}\}\min\{n,\epsilon_1^{-1}\epsilon_2^{-1/2}\})$. By this, \eqref{eq:alg1-oper-kthAL} and $\overline{K}_{\epsilon_1}=\cO(|\log\epsilon_1|^2)$ (see Remark \ref{order-out-itera-penalty}), we conclude that statement (ii) of Theorem \ref{thm:total-iter-cmplxity} holds.
\end{proof}

Next, we provide a proof of Theorem \ref{thm:total-iter-cmplxity2}. To proceed,
we first observe from Assumptions~\ref{asp:lowbd-knownfeas}(c) and \ref{asp:LICQ} that there exist ${U^f_g}>0$, $U^c_g>0$ and $\sigma>0$ such that
\begin{equation}\label{bound:fx-sigma-cx}
\|\nabla f(x)\|\le{U^f_g},\quad \|\nabla c(x)\|\le U^c_g,\quad \lambda_{\min}(\nabla c(x)^T \nabla c(x))\ge\sigma^2,\quad \forall x\in\cS(\bdelta_f,\bdelta_c).
\end{equation}
We next establish several technical lemmas that will be used shortly.

\begin{lemma}\label{lem:multiplier-update-without-proj}
Suppose that Assumptions \ref{asp:lowbd-knownfeas} and \ref{asp:LICQ} hold and that $\rho_0$ is sufficiently large
such that $\delta_{f,1} \le \bdelta_f$ and  $\delta_{c,1} \le \bdelta_c$, where $\delta_{f,1}$ and $\delta_{c,1}$ are defined in \eqref{def:delta0c-rhobar1xxx}.
Let $\{(x^k,\lambda^k,\rho_k)\}$ be generated by Algorithm \ref{alg:2nd-order-AL-nonconvex}. Suppose that
\begin{equation}\label{rho-bd2}
\rho_k\ge\max\{\Lambda^2(2\bdelta_f)^{-1},2(\fh-\fl+\gamma)\bdelta_c^{-2}+2\Lambda\bdelta_c^{-1}+2\gamma,2(U_g^f+U_g^c\Lambda+1)(\sigma\epsilon_1)^{-1}\}
\end{equation}
for some $k\ge0$, where $\gamma$, $\fh$, $\fl$, $\delta_f$ and $\delta_c$ are given in Assumption \ref{asp:lowbd-knownfeas}, and ${U^f_g}$, $U_g^c$ and $\sigma$ are given in \eqref{bound:fx-sigma-cx}. Then it holds that $\|c(x^{k+1})\|\le\epsilon_1$.
\end{lemma}

\begin{proof}
By \eqref{rho-bd2} and $\|\lambda^k\|\le\Lambda$ (see step \ref{algstep:proj-multiplier} of Algorithm \ref{alg:2nd-order-AL-nonconvex}), one can see that
\[
\rho_k\ge\max\{\|\lambda^k\|^2(2\bdelta_f)^{-1},2(\fh-\fl+\gamma)\bdelta_c^{-2}+2\|\lambda^k\|\bdelta_c^{-1}+2\gamma\}.
\]
 Using this, \eqref{hbd}, the first relation in \eqref{algstop:1st-order}, and Lemma \ref{tech-1}(iii) and (iv) with $(x,\lambda,\rho,\tdf,\tdc)=(x^{k+1},\lambda^k,\rho_k,\bdelta_f,\bdelta_c)$,  we obtain that $f(x^{k+1})\le \fh+\bdelta_f$ and $\|\tilde{c}(x^{k+1})\|\le\bdelta_c$. In addition, recall from $\|c(z_{\epsilon_1})\|\le1$ and the definition of $\tilde{c}$ in
\eqref{model:equa-cnstr-pert}
that $\|c(x^{k+1})\|\le1+\|\tilde{c}(x^{k+1})\|$. These together with \eqref{nearly-feas-level-set} show that $x^{k+1}\in\cS(\delta_f,\delta_c)$. It then follows from \eqref{bound:fx-sigma-cx} that $\|\nabla f(x^{k+1})\|\le{U^f_g}$, $\|\nabla c(x^{k+1})\|\le U^c_g$, and $\lambda_{\min}(\nabla c(x^{k+1})^T\nabla c(x^{k+1}))\ge\sigma^2$. By $\|\nabla f(x^{k+1})\|\le{U^f_g}$, $\|\nabla c(x^{k+1})\|\le U^c_g$, $\tau_k^g\le1$, $\|\lambda^k\|\le\Lambda$,
\eqref{model:equa-cnstr-pert}
and \eqref{algstop:1st-order}, one has
\begin{align}
&\rho_k\|\nabla c(x^{k+1})\tilde{c}(x^{k+1})\| \le\|\nabla f(x^{k+1})+\nabla c(x^{k+1})\lambda^k\| + \|\nabla_x \widetilde{\cL}(x^{k+1},\lambda^k;\rho_k)\|\nn\\
&\overset{\eqref{algstop:1st-order}}{\le} \|\nabla f(x^{k+1})\|+\|\nabla c(x^{k+1})\|\|\lambda^k\|+\tau_k^g \le U^f_g + U_g^c\Lambda +1. \label{grad-bound}
\end{align}
In addition, note that $\lambda_{\min}(\nabla c(x^{k+1})^T\nabla c(x^{k+1}))\ge\sigma^2$ implies that $\nabla c(x^{k+1})^T \nabla c(x^{k+1})$ is invertible.
Using this fact and \eqref{grad-bound}, we obtain
\begin{align}\label{ineq:1st-order-lambda-bound}
\|\tilde{c}(x^{k+1})\| &\le\|(\nabla c(x^{k+1})^T \nabla c(x^{k+1}))^{-1}\nabla c(x^{k+1})^T \|\|\nabla c(x^{k+1})\tilde{c}(x^{k+1})\|\nonumber \\
&= \lambda_{\min}(\nabla c(x^{k+1})^T\nabla c(x^{k+1}))^{-\frac12}\|\nabla c(x^{k+1})\tilde{c}(x^{k+1})\| \overset{\eqref{grad-bound}}{\le} (U^f_g + U_g^c\Lambda +1)/(\sigma\rho_k).
\end{align}
We also observe from \eqref{rho-bd2} that $\rho_k\ge2({U^f_g}+U_g^c\Lambda +1)(\sigma\epsilon_1)^{-1}$, which along with \eqref{ineq:1st-order-lambda-bound} proves $\|\tilde{c}(x^{k+1})\|\le\epsilon_1/2$. Combining this with the definition of $\tilde{c}$ in \eqref{model:equa-cnstr-pert}
and $\|c(z_{\epsilon_1})\|\le\epsilon_1/2$, we conclude that $\|c(x^{k+1})\|\le\epsilon_1$ holds as desired.
\end{proof}

The next lemma provides a stronger upper bound for $\{\rho_k\}$ than the one in Theorem \ref{thm:out-itr-cmplxity-1}.

\begin{lemma}\label{cor:improved-outer-iteration-cmplxty}
Suppose that Assumptions \ref{asp:lowbd-knownfeas} and \ref{asp:LICQ} hold and that $\rho_0$ is sufficiently large such that $\delta_{f,1} \le \bdelta_f$ and  $\delta_{c,1} \le \bdelta_c$, where $\delta_{f,1}$ and $\delta_{c,1}$ are defined in \eqref{def:delta0c-rhobar1xxx}. Let $\{\rho_k\}$ be generated by Algorithm \ref{alg:2nd-order-AL-nonconvex} and
\begin{equation}\label{bound-barrho2}
\tilde{\rho}_{\epsilon_1}:=\max\{\Lambda^2(2\bdelta_f)^{-1},2(\fh-\fl+\gamma)\bdelta_c^{-2}
+2\Lambda\bdelta_c^{-1}+2\gamma,2(U_g^f+U_g^c\Lambda+1)(\sigma\epsilon_1)^{-1},2\rho_0\},
\end{equation}
where $\gamma$, $\fh$, $\fl$, $\delta_f$ and $\delta_c$ are given in Assumption \ref{asp:lowbd-knownfeas},  and ${U^f_g}$, $U_g^c$ and $\sigma$ are given in \eqref{bound:fx-sigma-cx}. Then $\rho_k\le{r}\tilde{\rho}_{\epsilon_1}$ holds for $0 \le k\le \overline{K}_{\epsilon_1}$, where $\overline{K}_{\epsilon_1}$ is defined in \eqref{number-outer-iteration}.
\end{lemma}

\begin{proof}
It follows from \eqref{bound-barrho2} that $\tilde{\rho}_{\epsilon_1} \ge 2\rho_0$. By this and similar arguments as for \eqref{T-epsilong}, one has $K_{\epsilon_1}\le\log(\tilde{\rho}_{\epsilon_1}\rho_0^{-1})/\log{r}+1$, where $K_{\epsilon_1}$ is defined in \eqref{T-epsilon-g}.
Using this, the update scheme for $\{\rho_k\}$, and similar arguments as for \eqref{rho:upper-bound-Tepsilong}, we obtain
\begin{equation}\label{bound-rhok-Tepsilon}
\max_{0 \le k \le K_{\epsilon_1}} \rho_k \le {r}\tilde{\rho}_{\epsilon_1}.
\end{equation}
If $\|c(x^{K_{\epsilon_1}+1})\|\le\epsilon_1$, it follows from \eqref{number-outer-iteration} that $\overline{K}_{\epsilon_1}=K_{\epsilon_1}$, which together with \eqref{bound-rhok-Tepsilon} implies that $\rho_k\le{r}\tilde{\rho}_{\epsilon_1}$ holds for $0 \le k\le \overline{K}_{\epsilon_1}$. On the other hand,  if $\|c(x^{K_{\epsilon_1}+1})\|>\epsilon_1$, it follows from \eqref{number-outer-iteration} that $\|c(x^{k+1})\|>\epsilon_1$ for $K_{\epsilon_1} \le k \le \overline{K}_{\epsilon_1}-1$. This together with Lemma \ref{lem:multiplier-update-without-proj} and \eqref{bound-barrho2} implies that for all $K_{\epsilon_1} \le k \le \overline{K}_{\epsilon_1}-1$,
\[
\rho_k <\max\{\Lambda^2(2\bdelta_f)^{-1},2(\fh-\fl+\gamma)\bdelta_c^{-2}+2\Lambda\bdelta_c^{-1}
+2\gamma,2(U_g^f+U_g^c\Lambda+1)(\sigma\epsilon_1)^{-1}\}\overset{\eqref{bound-barrho2}}{\le}\tilde{\rho}_{\epsilon_1}.
\]
By this, \eqref{bound-rhok-Tepsilon}, and $\rho_{\overline{K}_{\epsilon_1}}\le{r}\rho_{\overline{K}_{\epsilon_1}-1}$, we also see that $\rho_k\le{r}\tilde{\rho}_{\epsilon_1}$ holds for $0 \le k\le \overline{K}_{\epsilon_1}$.
\end{proof}
	

\begin{proof}[Proof of Theorem \ref{thm:total-iter-cmplxity2}]
Notice from \eqref{bound-barrho2} and Lemma \ref{cor:improved-outer-iteration-cmplxty} that $\tilde{\rho}_{\epsilon_1}=\cO(\epsilon_1^{-1})$ and $\rho_k \le {r}\tilde{\rho}_{\epsilon_1}$, which yield $\rho_k=\cO(\epsilon_1^{-1})$.  The conclusion of Theorem \ref{thm:total-iter-cmplxity2} then follows from this and the same arguments as for the proof of Theorem \ref{thm:total-iter-cmplxity} with $\rho_k=\cO(\epsilon_1^{-2})$ replaced by $\rho_k=\cO(\epsilon_1^{-1})$.
\end{proof}	
	
\section{Future work}\label{sec:cr}

There are several possible future studies on this work. First, it would be interesting to extend our AL method to seek an approximate SOSP of nonconvex optimization with inequality or more general constraints. Indeed, for nonconvex optimization with inequality constraints, one can reformulate it as an equality constrained problem using squared slack variables (e.g., see \cite{B97}). It can be shown that an SOSP of the latter problem induces a weak SOSP of the original problem and also linear independence constraint qualification holds for the latter problem if it holds for the original problem. As a result, it is promising to find an approximate weak SOSP of an inequality constrained problem by applying our AL method to the equivalent equality constrained problem. Second, it is worth studying whether the enhanced complexity results in Section~\ref{sec:AL-modified} can be derived under weaker constraint qualification (e.g., see \cite{AHSS12}). Third, the development of our AL method is based on a strong assumption that a nearly feasible solution of the problem is known. It would make the method applicable to a broader class of problems if such an assumption could be removed by modifying the method possibly through the use of infeasibility detection techniques (e.g., see \cite{BCW14}). Lastly, more numerical studies would be helpful to further improve our AL method from a practical perspective.


\section*{Appendix}
\appendix

\section{A capped conjugate gradient method}\label{appendix:capped-CG}
In this part we present the capped CG method proposed in \cite[Algorithm~1]{RNW18} for finding either an approximate solution to the linear system \eqref{indef-sys} or a sufficiently negative curvature direction of the associated matrix $H$, which has been briefly discussed in Section \ref{sbsc:main-cmpnts}. Its details can be found in \cite[Section~3.1]{RNW18}.
\begin{algorithm}[h]
\caption{A capped conjugate gradient method}
\label{alg:capped-CG}
{\footnotesize
\begin{algorithmic}
\State \noindent\textit{Inputs}: symmetric matrix $H\in\bR^{n\times n}$, vector $g\neq0$, damping parameter $\varepsilon\in(0,1)$, desired relative accuracy $\zeta\in(0,1)$.
\State \textit{Optional input:} scalar $U\ge0$ (set to $0$ if not provided).
\State \textit{Outputs:} d$\_$type, $d$.
\State \textit{Secondary outputs:} final values of $U,\kappa,\widehat{\zeta},\tau,$ and $T$.
\State Set
\begin{equation*}
\textstyle
\bar{H}:=H+2\varepsilon I,\quad \kappa:=\frac{U+2\varepsilon}{\varepsilon},\quad\widehat{\zeta}:=\frac{\zeta}{3\kappa},\quad\tau:=\frac{\sqrt{\kappa}}{\sqrt{\kappa}+1},\quad T:=\frac{4\kappa^4}{(1-\sqrt{\tau})^2},
\end{equation*}
$y^0\leftarrow 0,r^0\leftarrow g,p^0\leftarrow -g, j\leftarrow 0$.
\If {$(p^0)^T \bar{H}p^0<\varepsilon\|p^0\|^2$}
\State Set $d\leftarrow p^0$ and terminate with d$\_$type = NC;
\ElsIf {\ $\|Hp^0\|>U\|p^0\|\ $}
\State Set $U\leftarrow\|Hp^0\|/\|p^0\|$ and update $\kappa,\widehat{\zeta},\tau, T$ accordingly;
\EndIf
\While{TRUE}
\State $\alpha_j\leftarrow (r^j)^T r^j/(p^j)^T\bar{H}p^j$; \{Begin Standard CG Operations\}
\State $y^{j+1}\leftarrow y^j+\alpha_jp^j$;
\State $r^{j+1}\leftarrow r^j+\alpha_j\bar{H}p^j$;
\State $\beta_{j+1}\leftarrow\|r^{j+1}\|^2/\|r^j\|^2$;
\State $p^{j+1}\leftarrow-r^{j+1}+\beta_{j+1}p^j$; \{End Standard CG Operations\}
\State $j\leftarrow j+1$;
\If {$\|Hp^j\|>U\|p^j\|$}
\State Set $U\leftarrow\|Hp^j\|/\|p^j\|$ and update $\kappa,\widehat{\zeta},\tau,T$ accordingly;
\EndIf
\If {\ $\|Hy^j\|>U\|y^j\|\ $}
\State Set $U\leftarrow\|Hy^j\|/\|y^j\|$ and update $\kappa,\widehat{\zeta},\tau,T$ accordingly;
\EndIf
\If {\ $\|Hr^j\|>U\|r^j\|\ $}
\State Set $U\leftarrow\|Hr^j\|/\|r^j\|$ and update $\kappa,\widehat{\zeta},\tau,T$ accordingly;
\EndIf
\If {$(y^j)^T\bar{H}y^j<\varepsilon\|y^j\|^2$}
\State Set $d\leftarrow y^j$ and terminate with d$\_$type = NC;
\ElsIf {\ $\|r^j\|\le\widehat{\zeta}\|r^0\|$}
\State Set $d\leftarrow y^j$ and terminate with d$\_$type = SOL;
\ElsIf{\ $(p^j)^T\bar{H}p^j<\varepsilon\|p^j\|^2$}
\State Set $d\leftarrow p^j$ and terminate with d$\_$type = NC;
\ElsIf {\ $\|r^j\|>\sqrt{T}\tau^{j/2}\|r^0\| $}
\State Compute $\alpha_j, y^{j+1}$ as in the main loop above;
\State Find $i\in\{0,\ldots,j-1\}$ such that
\[
(y^{j+1}-y^i)^T\bar{H}(y^{j+1}-y^i)<\varepsilon\|y^{j+1}-y^i\|^2;
\]
\State Set $d\leftarrow y^{j+1}-y^i$ and terminate with d$\_$type = NC;
\EndIf
\EndWhile
\end{algorithmic}
}
\end{algorithm}

The following theorem presents the iteration complexity of Algorithm~\ref{alg:capped-CG}.

\begin{theorem}[{\bf iteration complexity of Algorithm \ref{alg:capped-CG}}]\label{lem:capped-CG}
Consider applying Algorithm~\ref{alg:capped-CG} with input $U = 0$ to the linear system~\eqref{indef-sys} with $g\neq 0$, $\varepsilon>0$, and $H$ being an $n\times n$ symmetric matrix. Then the number of iterations of Algorithm~\ref{alg:capped-CG} is $\widetilde{\cO}(\min\{n,\sqrt{\|H\|/\varepsilon}\})$.
\end{theorem}

\begin{proof}
From \cite[Lemma~1]{RNW18}, we know that the number of iterations of Algorithm \ref{alg:capped-CG} is bounded by $\min\{n,J(U,\varepsilon,\zeta)\}$, where $J(U,\varepsilon,\zeta)$ is the smallest integer $J$ such that $\sqrt{T}\tau^{J/2}\le\widehat{\zeta}$, with $U,\widehat{\zeta},T$ and $\tau$ being the values returned by Algorithm \ref{alg:capped-CG}. In addition, it was shown in \cite[Section 3.1]{RNW18} that
$J(U,\varepsilon,\zeta)\le\left\lceil \left(\sqrt{\kappa}+\frac{1}{2}\right)\ln\left(\frac{144(\sqrt{\kappa}+1)^2\kappa^6}{\zeta^2}\right)\right\rceil$,
where $\kappa={\cO}(U/\varepsilon)$ is an output by Algorithm \ref{alg:capped-CG}. Then one can see that $J(U,\varepsilon,\zeta)=\widetilde{\cO}(\sqrt{U/\varepsilon})$. Notice from
Algorithm~\ref{alg:capped-CG} that the output $U\le\|H\|$. Combining these, we obtain the conclusion as desired.
\end{proof}

\section{A randomized Lanczos based minimum eigenvalue oracle} \label{appendix:meo}
In this part we present the randomized Lanczos method proposed in \cite[Section~3.2]{RNW18}, which can be used as a minimum eigenvalue oracle for Algorithm~\ref{alg:NCG}. As briefly discussed in Section~\ref{sbsc:main-cmpnts}, this oracle outputs either a sufficiently negative curvature direction of $H$ or a certificate that $H$ is nearly positive semidefinite with high probability. More detailed motivation and explanation of it can be found in \cite[Section~3.2]{RNW18}.

\begin{algorithm}[h]
\caption{A randomized Lanczos based minimum eigenvalue oracle}
\label{pro:meo}
\noindent\textit{Input}: symmetric matrix $H\in\bR^{n\times n}$, tolerance $\varepsilon>0$, and probability parameter $\delta\in(0,1)$.\\
\noindent\textit{Output}: a sufficiently negative curvature direction $v$ satisfying $v^THv\le-\varepsilon/2$ and $\|v\|=1$; or a certificate that $\lambda_{\min}(H)\ge-\varepsilon$ with probability at least  $1-\delta$.\\
Apply the Lanczos method \cite{KW92LR} to estimate $\lambda_{\min}(H)$ starting with a random vector uniformly generated on the unit sphere, and run it for at most 
\begin{equation}\label{N-iter} 
N(\varepsilon,\delta):=\min\left\{n,1+\left\lceil\frac{\ln(2.75n/\delta^2)}{2}\sqrt{\frac{\|H\|}{\varepsilon}}\right\rceil\right\}
\end{equation}
iterations. If a unit vector $v$ with $v^THv \le -\varepsilon/2$ is found at some iteration, terminate immediately and return $v$.
\end{algorithm}

The following theorem justifies that Algorithm~\ref{pro:meo} is a suitable minimum eigenvalue oracle for Algorithm \ref{alg:NCG}. Its proof is identical to that of \cite[Lemma~2]{RNW18} and thus omitted.

\begin{theorem}[{\bf iteration complexity of Algorithm~\ref{pro:meo}}]\label{rand-Lanczos}
Consider Algorithm~\ref{pro:meo} with tolerance $\varepsilon>0$, probability parameter $\delta\in(0,1)$, and symmetric matrix $H\in\bR^{n\times n}$ as its input. Then it either finds a sufficiently negative curvature direction $v$ satisfying $v^THv\le-\varepsilon/2$ and $\|v\|=1$ or certifies that $\lambda_{\min}(H)\ge-\varepsilon$ holds with probability at least  $1-\delta$  in at most $N(\varepsilon,\delta)$ iterations, where $N(\varepsilon,\delta)$ is defined in \eqref{N-iter}.
\end{theorem}

Notice that $\|H\|$ is required in Algorithm~\ref{pro:meo}. In general, computing $\|H\|$  may not be cheap when $n$ is large. Nevertheless, $\|H\|$ can be efficiently estimated via a randomization scheme with high confidence (e.g., see the discussion in \cite[Appendix~B3]{RNW18}).
\end{document}